\documentclass[final,onefignum,onetabnum]{siamart190516}

\usepackage{lipsum}
\usepackage{graphicx}
\usepackage{epstopdf}
\ifpdf
\DeclareGraphicsExtensions{.eps,.pdf,.png,.jpg}
\else
\DeclareGraphicsExtensions{.eps}
\fi

\usepackage{epsf}
\usepackage{geometry}
\usepackage{listings} 
\usepackage{algorithmic,algorithm}
\usepackage{multirow}
\usepackage{enumerate}
\usepackage{color}
\usepackage{booktabs}
\usepackage{comment}

\usepackage{amsmath, amsfonts, amssymb,mathrsfs}
\usepackage{verbatim}
\usepackage{bm}

\newsiamremark{remark}{Remark}
\newsiamremark{hypothesis}{Hypothesis}
\crefname{hypothesis}{Hypothesis}{Hypotheses}
\newsiamthm{claim}{Claim}

\usepackage{amsopn}

\newcommand{\ddiv}{\operatorname{div}}
\headers{Mixed-dimensional Problems with High Dimensional Gap}{E.~Hodneland, X.~Hu, and J.~M.~Nordbotten}

\title{Modeling, well-posedness and discretization for a class of models for mixed-dimensional problems with high dimensional gap
\thanks{
\funding{E. Hodneland is supported through Norwegian Research Council (NRC) grant 262203. The work of 
X. Hu is partially supported by the National Science Foundation under grant DMS-1812503 and CCF-1934553. 
The work of J. M. Nordbotten is partially supported by NRC grant number 250223.}
}
}

\author{Erlend Hodneland\thanks{Norwegian Research Centre, Bergen, Norway; Mohn Medical Imaging and Visualization Centre, Department of Radiology, Haukeland Universitetssykehus, Bergen, Norway
		(\email{erlend.hodneland@norceresearch.no}).}
	\and Xiaozhe Hu\thanks{Department of Mathematics, Tufts University, Medford, MA 02155, USA
		(\email{Xiaozhe.Hu@tufts.edu}).}
	\and Jan Martin Nordbotten \thanks{Department of Mathematics, University of Bergen, Bergen, Norway (\email{jan.nordbotten@math.uib.no}).}
}

\begin{document}

\maketitle

\begin{abstract}
In this work, we illustrate the underlying mathematical structure of mixed-dimensional models arising from the composition of graphs and continuous domains. Such models are becoming popular in applications, in particular, to model the human vasculature. We first discuss the model equations in the strong form which describes the conservation of mass and Darcy's law in the continuum and network as well as the coupling between them.  By introducing proper scaling, we propose a weak form that avoids degeneracy.  Well-posedness of the weak form is shown through standard Babu\v{s}ka-Brezzi theory. We also develop the mixed formulation finite-element method and prove its well-posedness.  A mass-lumping technique is introduced to derive the two-point flux approximation type discretization as well, due to its importance in applications.  Based on the Babu\v{s}ka-Brezzi theory, error estimates can be obtained for both the finite-element scheme and the TPFA scheme. We also discuss efficient linear solvers for discrete problems. Finally, we present some numerical examples to verify the theoretical results and demonstrate the robustness of our proposed discretization schemes.  
\end{abstract}

\begin{keywords}
	Mixed-dimensional problems, mixed-formulation finite-element method \LaTeX
\end{keywords}

\begin{AMS}
	65N30, 65N15, 65N08, 65N22
\end{AMS}

\section{Introduction} \label{sec:intro}
Coupled fluid flow in networks and porous domains arise in various applications, including blood flow in the human body as well as wells in geological applications.  Such models are referred to as mixed-dimensional when the network flow is simplified to a family of 1D domains along with the network edges\footnote{Some authors refer to this problem as multiscale (see e.g. \cite{DangeloQuarteroni2008,KopplVidottoWohlmuthZunino2018}), however, we prefer the nomenclature mixed-dimensional to avoid confusion with equidimensional multiscale methods such as are encountered in (numerical) homogenization problems.}. Moreover, when the coupling between the network and the domain exceeds two topological dimensions, the model is referred to as having a high dimensional gap~\cite{Nordbotten2019,LaurinoZunino2019}. A high dimensional gap thus arises when the flow in the network is connected to a domain of dimension $d\geq 2$ through its leaf nodes, or when the flow in the network is connected to a domain of dimension $d\geq 3$ through its edges. 

\begin{figure}
    \centering
    \includegraphics{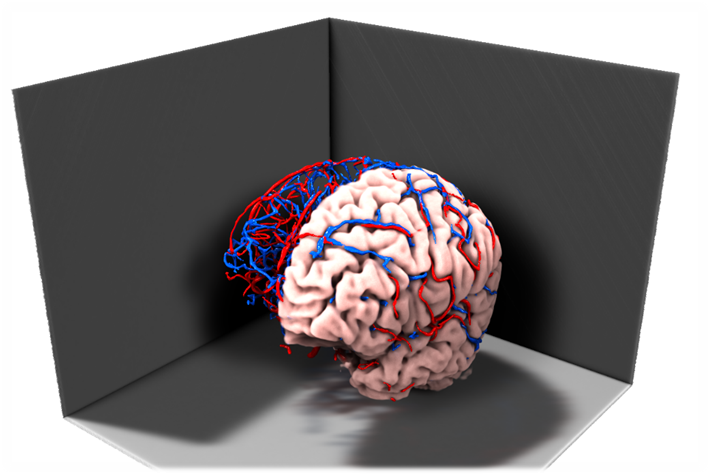}
    \caption{Illustration of a characteristic mixed-dimensional geometry associated with blood flow in the brain. This illustration is based on the the data-set used in the full-brain simulation study in section \ref{sec:Case3}. The arterial tree is indicated in red, and the venous tree in blue. Note the complex geometry of the outer boundary of the brain (i.e. the domain $\Omega$).}
    \label{fig:brain-figure}
\end{figure}

In this paper, we consider the problem composed of flow in one or more trees, coupled with a (porous) domain.  This setting is motivated by blood flow in the brain, wherein the networks are the arterial and venous trees, and the domain is the sub-resolution capillary bed. This context is shown in Figure \ref{fig:brain-figure}, which we will return to in the numerical results. Recognizing that the leaf nodes in the tree (referred to as "terminals" hereafter) are in applications an artifact of limited imaging resolution, we consider in our equations a mesoscale model wherein fluid is distributed into the porous domain in a support region near the terminals. Such models have recently been introduced in \cite{HodnelandHansonSaevareidNaevdalLundervoldSolteszovaMunthe-KaasDeistungReichenbachNordbotten2019} and also considered in \cite{koch2020modeling,shipley2020hybrid}, and are attractive also from a mathematical perspective, as they avoid the singularities which otherwise characterize the coupled equations. In this work, we will not adopt the precise models used in \cite{HodnelandHansonSaevareidNaevdalLundervoldSolteszovaMunthe-KaasDeistungReichenbachNordbotten2019,koch2020modeling} directly, as they consider an explicitly given structure of fluid distribution between the network and the porous domain. In contrast, we will use a more canonical formulation, where the flow resistance is given, and the fluid distribution from the terminal is calculated. 

Previous mathematical analysis of models with high dimensional gap has to a large extent been focused on how to handle the singularities arising when the coupling is "point-wise" between the network and the domain (see e.g. \cite{DangeloQuarteroni2008,KopplVidottoWohlmuth2016,GjerdeKumarNordbottenWohlmuth2019}). In contrast, the model discussed herein has to our knowledge not been subjected to mathematical analysis before. In the absence of singularities, we exploit in this paper the framework recently developed for problems with small dimensional gap~\cite{BoonNordbottenVatne2019}, and define mixed-dimensional variables and operators for the coupled problem. Together with appropriately defined integration and inner products, we then observe that we have available tools such as a mixed-dimensional Stokes' theorem, integration by parts, and Hilbert spaces. This forms the building blocks for our well-posedness results and numerical analysis. 

The main results of the paper are thus as follows: 

\begin{itemize}
    \item A general, non-singular model for a class of problems with a large dimensional gap. 
    \item Well-posedness theory for both the continuous and finite-dimensional problem. 
    \item Convergence results for mixed finite-element approximation and a finite volume variant. 
    \item Numerical validation and application to a high-resolution data-set of a real human brain. 
\end{itemize}

We structure the paper as follows. In Section~\ref{sec:model} we present the model equations in both strong and weak forms and show well-posedness. In Section~\ref{sec:FEM} and~\ref{sec:TPFA} we state and analyze the finite-element and finite volume approximations, respectively. The theoretical results are validated in Section~\ref{sec:numerics}. Finally, we give some conclusions in Section~\ref{sec:conclusion}.

\section{Model Equations} \label{sec:model}
In this section, we discuss the basic geometric setup and model equations for coupled network-Darcy flow in brain.  We will both introduce the strong form and then derive the weak form by introducing proper spaces. 

\subsection{Geometry}
We are concerned with a domain~$\Omega \subset \mathbb{R}^n$ (which models the capillaries). In addition, we are concerned with a finite collection of rooted trees~$\mathcal{T}$ with node (vertex) set~$\mathcal{N}_{\mathcal{T}}$ and edge set~$\mathcal{E}_{\mathcal{T}}$ (which model resolved arteries and veins).  The arterial and venous trees are considered disjoint and, therefore, form a forest $\mathcal{F}$ with node set $\mathcal{N} = \cup_{\mathcal{T}\in\mathcal{F}} \, \mathcal{N}_\mathcal{T}$ and edge set $\mathcal{E}= \cup_{\mathcal{T} \in \mathcal{F}} \, \mathcal{E}_{\mathcal{T}}$. We will refer to the composite (mixed-dimensional) problem domain of both $\Omega$ and $\mathcal{F}$ as the disjoint union $\mathfrak{B}=\Omega\sqcup\mathcal{F}$.  

We further distinguish the nodes of the forest as follows. The node set $\mathcal{N}$ can be subdivided into three disjoint subsets, the first and last of which are assumed to be non-empty: root nodes~$\mathcal{N}_R$, interior nodes~$\mathcal{N}_I$, and terminal nodes~$\mathcal{N}_T$. Note that $\mathcal{N} = \mathcal{N}_R \cup \mathcal{N}_I \cup \mathcal{N}_T$ and we use $\mathcal{N}_{\mathcal{T}, R} = \mathcal{N}_{\mathcal{T}} \cap \mathcal{N}_R$, $\mathcal{N}_{\mathcal{T}, I} = \mathcal{N}_{\mathcal{T}} \cap \mathcal{N}_I$, and $\mathcal{N}_{\mathcal{T}, T} = \mathcal{N}_{\mathcal{T}} \cap \mathcal{N}_T$ to denote the root nodes, interior nodes, and the terminal nodes of a given tree $\mathcal{T}$, respectively.  Naturally, we also have $\mathcal{N}_{\mathcal{T}} = \mathcal{N}_{\mathcal{T},R} \cup \mathcal{N}_{\mathcal{T},I} \cup \mathcal{N}_{\mathcal{T},T}$. We further divided the root nodes $\mathcal{N}_{R}$ into two disjoint sets $\mathcal{N}_{D}$, which consists of the Dirichlet root nodes, and $\mathcal{N}_{N}$, which consists of the Neumann root nodes. The Dirichlet root nodes will be treat explicitly as Dirichlet boundary conditions and the Neumann root nodes will be implicitly handled through the right-hand side of the conservation laws on the graph. Following the same convention, $\mathcal{N}_{\mathcal{T}, D}$ and $\mathcal{N}_{\mathcal{T},N}$ denotes the Dirichlet or Neumann root nodes of a given tree $\mathcal{T}$, respectively.  Note that each tree can only have one root. Therefore, we can subdivide the forest into two disjoint sub-forests, i.e., Dirichlet rooted forest $\mathcal{F}_{D}$, which contains all the Dirichlet rooted trees $\mathcal{T}_{D}$, and Neumann rooted forest $\mathcal{F}_{N}$, which contains all the Neumann rooted trees $\mathcal{T}_{N}$.  Naturally, $\mathcal{N}_{\mathcal{F}_D} = \cup_{\mathcal{T}\in\mathcal{F}_D} \mathcal{N}_{\mathcal{T}}$ and $\mathcal{N}_{\mathcal{F}_N} = \cup_{\mathcal{T}\in\mathcal{F}_N} \mathcal{N}_{\mathcal{T}}$.  Furthermore, we define, $\mathcal{N}_{\mathcal{F}_D, R} = \mathcal{N}_{\mathcal{F}_D} \cap \mathcal{N}_R$, $\mathcal{N}_{\mathcal{F}_D, I} = \mathcal{N}_{\mathcal{F}_D} \cap \mathcal{N}_I$, $\mathcal{N}_{\mathcal{F}_D, T} = \mathcal{N}_{\mathcal{F}_D} \cap \mathcal{N}_T$, $\mathcal{N}_{\mathcal{F}_N, R} = \mathcal{N}_{\mathcal{F}_N} \cap \mathcal{N}_R$, $\mathcal{N}_{\mathcal{F}_N, I} = \mathcal{N}_{\mathcal{F}_N} \cap \mathcal{N}_I$, and $\mathcal{N}_{\mathcal{F}_N, T} = \mathcal{N}_{\mathcal{F}_N} \cap \mathcal{N}_T$.
We denote the set of the neighbors of the node $i$ as $\mathcal{N}_i$ and the set of all the edges meeting at~$i\in \mathcal{N}$  as~$\mathcal{E}_i$.  Note that, $|\mathcal{N}_i| = 1$ and $|\mathcal{E}_i| = 1$, if $i \in \mathcal{N}_R \cup \mathcal{N}_T$. These concepts are illustrated for $n=2$ in Figure \ref{fig:schematic}. 

\begin{figure}
    \centering
    \includegraphics[width=15cm]{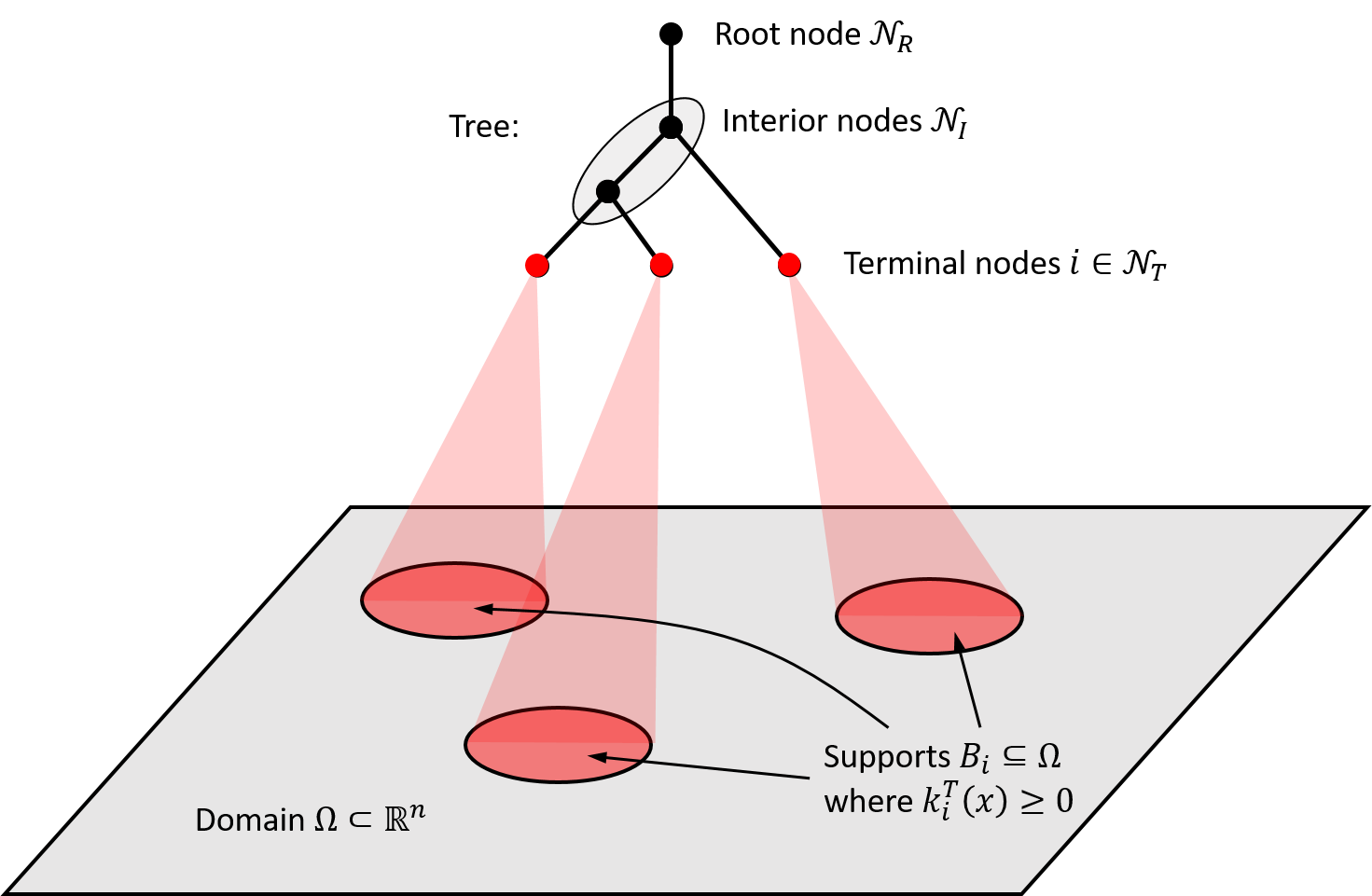}
    \caption{Schematic illustrating the mixed-dimensional geometry $\mathfrak{B}$, and its subdivision into a continuous domain and a forest. Also shown is the coupling between the trees and the model domain.}
    \label{fig:schematic}
\end{figure}

\subsection{Strong From}
As primary variables we choose the domain pressure potential~$p^D(\bm{x}):\Omega \mapsto \mathbb{R}$ and the node pressure potentials~$p^{N}:\mathcal{N} \mapsto \mathbb{R}$. Furthermore, we consider the fluid mass fluxes denoted in the domain as~$\bm{q}^D(\bm{x}): \Omega \mapsto \mathbb{R}^n$, fluid mass flow from node~$i$ to ~$j$ denoted~$q^{N}_{i,j}: \mathcal{N} \times \mathcal{N} \mapsto \mathbb{R}$ and fluid mass flow transferring from terminal node~$i$ to point~$\bm{x}$ denoted~$q_i^T(\bm{x}):\Omega \mapsto \mathbb{R}$.  This last variable models the flow in \emph{unresolved} arteries and veins, and is a novel component our this work. 

First, we consider the model equations for mass conservation and they are given as follows based on the above definitions and notation. 
\begin{align}
&\text{(Conservation of mass in brain tissue)} \quad \nabla \cdot \bm{q}^D - \sum_{i\in \mathcal{N}_T} q_i^T = r^D, \quad \text{in} \ \Omega \label{eqn:conservation-brain}	\\
&\text{(Conservation of mass at interior nodes)} \quad -\sum_{j \in \mathcal{N}_i} q^N_{j,i} = r_i^N, \quad \text{for all} \ i \in \mathcal{N}_I \cup \mathcal{N}_N \label{eqn:conservation-int-node} \\
&\text{(Conservation of mass at terminal nodes)}\int_{\Omega} q_i^T(\bm{x}) \mathrm{d} \bm{x} - q^N_{\mathcal{N}_i,i} = r_i^N \quad \text{for all} \ i \in \mathcal{N}_T \label{eqn:conservation-term-node}
\end{align}
Here, the signs in~\eqref{eqn:conservation-brain}-\eqref{eqn:conservation-term-node} are chosen such that the right-hand-side terms represent sources added to the system.  Moreover, although both $q^N_{i,j}$ and $q^N_{j,i}$ are used in~\eqref{eqn:conservation-int-node} for notational convenience, they should be understood as one unknown with a sign difference.

Next we verify the global conservation of mass based on~\eqref{eqn:conservation-brain}-\eqref{eqn:conservation-term-node} as follows,
\begin{align*}
\text{(Stokes' Theorem)} \quad  \int_{\partial \Omega} \bm{q}^D \cdot \bm{n} \, \mathrm{d} \bm{x} &= \int_{\Omega} \nabla \cdot \bm{q}^D \, \mathrm{d} \bm{x} \\
\text{(By \eqref{eqn:conservation-brain})} & = \sum_{i\in \mathcal{N}_T} \int_{\Omega} q_i^T(\bm{x}) \, \mathrm{d} \bm{x} + \int_{\Omega} r^D \, \mathrm{d}\bm{x} \\
\text{(By \eqref{eqn:conservation-term-node})} &= \sum_{i \in \mathcal{N}_T} q_{\mathcal{N}_i,i}^N + \sum_{i \in \mathcal{N}_T} r^N_i + \int_{\Omega} r^D \, \mathrm{d}\bm{x}\\
\text{(By~\eqref{eqn:conservation-int-node})} &= \sum_{i \in \mathcal{N}_D} q_{i,\mathcal{N}_i}^N + \sum_{i\in \mathcal{N}_{N} \cup \mathcal{N}_I \cup \mathcal{N}_T} r_i^{N} + \int_{\Omega} r^D \, \mathrm{d}\bm{x}.
\end{align*}
Where the last step is also known as the Graph-Stokes' Theorem, which is the counterpart of the Stokes' Theorem on graphs. 

We now propose constitutive laws for the flow. As our exposition primarily is concerned with the geometric complexity, we herein only consider linear constitutive laws, although it is reasonable that non-linear extensions may be required in applications (see e.g. \cite{van2011pulse}). We therefore introduce material parameters, all of which are assumed to be non-negative (precise bounds are given later).  For each edge $e(i,j) \in \mathcal{E}$, we assign a conductivity $k_{e(i,j)}^N$, which can be considered as the edge weights in certain sense. In the domain, for each $\bm{x} \in \Omega$ we assign a permeability tensor $k^D(\bm{x}): \Omega \mapsto \mathbb{R}^{n \times n}$.  For each terminal node $i \in \mathcal{N}_T$, we assign connectivity function $k_i^T(\bm{x}): \Omega \mapsto \mathbb{R}$.  Now based on the assumption that the potential flow is linear, we have the following constitutive laws. 
\begin{align}
&\text{(Potential flow in brain (Darcy))} \quad  \bm{q}^D = -k^D(\bm{x}) \nabla p^D, \quad \text{in} \ \Omega, \label{eqn:potential-flow-brain} \\
&\text{(Potential flow in network (Poiseuille))} \quad q_{i,j}^N = -k_{e(i,j)}^N\left( p_j^N - p_i^N \right), \quad \text{for} \ e(i,j) \in \mathcal{E}, \label{eqn:pontential-flow-network} \\
&\text{(Potential flow from network to brain)} \quad q_i^T(\bm{x}) = -k_i^T(\bm{x}) \left( p^D - p_i^N \right), \quad \text{for} \ i \in \mathcal{N}_T. \label{eqn:potential-flow-coupling}
\end{align}
The coefficient functions~$k_i^T(\bm{x})$, $i \in \mathcal{N}_T$, represent redistribution in a small region around the terminal node~$i$, thus, can be assumed to have compact support in some domain~$B_i \subseteq \Omega$.

\begin{remark}
In practice the characteristic length scale of $B_i$ is comparable to the distance to the nearest neighbor, i.e.
\begin{equation} \label{eqn:diam-Bi}
\operatorname{diam}(B_i) = \mathcal{O}\left( \min_{j \in \mathcal{N}_T} |\bm{x}_i - \bm{x}_j| \right).
\end{equation}
Moreover, the grid is frequently given by the voxel resolution of the image and the terminals are due to a finite resolution effect, and thus 
\begin{equation} \label{eqn:xi-xj}
 \min_{j \in \mathcal{N}_T} |\bm{x}_i - \bm{x}_j|  =  \mathcal{O}\left(h \right),
\end{equation}
where $h$ is the mesh size. The constants hidden in the $\mathcal{O}$ notation in~\eqref{eqn:diam-Bi} and~\eqref{eqn:xi-xj} are usually between $2$ to $10$ in practical applications that we are interested in.  Consequentially, $q_i^T(\bm{x})$ also is compactly supported in~$B_i$.  While these considerations could be applied to further refine some of the constants in the proofs below, we will not exploit these details in this paper. 
\end{remark}

In addition to the conservation laws and constitutive laws, we also need boundary conditions to close the system. For the sake of simplicity, we only consider the case of homogeneous Neumann data on~$\partial \Omega$ and the Dirichlet root nodes $\mathcal{N}_D$, i.e.,
\begin{equation}\label{eqn:BC}
\bm{q}^D \cdot \bm{n} = 0 \quad \text{on} \ \partial \Omega \qquad  \text{and} \qquad p^N_i = 0, \quad i \in \mathcal{N}_D.
\end{equation}
We want to point out that our results and analysis below also hold for other types of boundary conditions as well only at the cost of extra notation. The choice of Neumann data on $\partial\Omega$ is in a sense the most difficult case, as the inf-sup proofs can be simplified considerably in the case where there is a measurable subset of the boundary with Dirichlet data. 

We close this subsection by the observation that by definition~$q^N_{i,j} = - q^N_{j,i}$. Therefore, although the total number of $q^N_{i,j}$ is $2|\mathcal{E}|$, we only use half of them as the unknowns, i.e., one unknown, $q^N_{i,j}$ or $q^N_{j,i}$, for each edge $e(i,j) \in \mathcal{E}$. The choice is arbitrary.  In this work, we choose the one follows the direction from the root node to the terminal nodes. This direction is also the assigned orientation of the corresponding edge $e(i,j) \in \mathcal{E}$ (i.e., if we choose $q^N_{i,j}$, which means the fluid mass flows from node $i$ to node $j$, the edge $e(i,j)$ is oriented such that it starts at node $i$ and ends at node $j$).  This allows us to define the following signed incidence matrix $\mathcal{G} \in \mathbb{R}^{|\mathcal{E}| \times |\mathcal{N}|}$, such that 
\begin{equation}\label{eqn:incidence-matrix-forest}
\mathcal{G}_{\ell, i} = 
\begin{cases}
1,  &  \text{if flow on edge $\ell$ starts at node $i$} \\
-1, &  \text{if flow on edge $\ell$ ends at node $i$}  \\
0,  & \text{otherwise}  
\end{cases}
\end{equation}
We want to point out that the signed incidence matrix represents a discrete gradient on the graph and its transpose serves as a discrete divergence.

\subsection{Mixed-dimensional formulation and scaling}
The model equations given above contain essentially three expressions of fluxes ($\bm{q}^D, q^T$ and $q^N$), and two expressions of potentials ($p^D$ and $p^N$). It will simplify the following exposition and analysis considerably to treat these as mixed-dimensional variables on $\mathfrak{B}$, on which we define mixed-dimensional operators. 

Therefore, let the mixed-dimensional pressure be denoted $\mathfrak{p}\ : \mathfrak{B}\rightarrow \mathbb{R}$, and defined as the doublet of pressures $\mathfrak{p}:=[p^D, p^N]$. Equivalently, the mixed-dimensional flux is defined as the triplet $\mathfrak{q}:=[\bm{q}^D, q^T, q^N]$. Now, we define the  mixed-dimensional divergence operator $\mathfrak{D} \cdot$ as follows, 
\begin{equation} \label{eqn:mixed-divergence1}
\mathfrak{D} \cdot \mathfrak{q} = \mathfrak{D} \cdot [\bm{q}^D, q^T, q^N] := [u^D, u^N],
\end{equation}
where
\begin{equation}
u^D := \nabla \cdot \bm{q}^D - \sum_{i\in \mathcal{N}_T} q_i^T  \quad \text{and} \quad u^N_i = 
\begin{cases}
 - \sum_{j \in \mathcal{N}_i} q_{j,i}^N, \quad i \in \  \mathcal{N}_I \cup \mathcal{N}_N,  \\
 \int_{B_i} q_i^T(\bm{x}) \, \mathrm{d} \bm{x} - q_{\mathcal{N}_i,i}^N, \quad i \in \mathcal{N}_T.
\end{cases}
\end{equation}

Similarly, we define the mixed-dimensional gradient $\mathfrak{D}$ as
\begin{equation} \label{eqn:mixed-gradient1}
\mathfrak{D} \mathfrak{p} = \mathfrak{D} [p^D, p^N] := [\bm{v}^D, v^T, v^N],
\end{equation}
where
\begin{equation}
\bm{v}^D := \nabla p^D\quad \text{and} \quad
v^T_i(\bm{x}) := p^D(\bm{x}) - p_i^N, \ i \in \mathcal{N}_T, 
\quad \text{and}
\quad v^N = \mathcal{G} p^N.
\end{equation}

In addition, we introduce the function $\mathfrak{K}$ which contains all the material functions $k^D$, $k_i^T(\bm{x})$, $i \in \mathcal{N}_T$, and $k_{e(i,j)}^N$, $e(i,j)\in \mathcal{E}$, in~\eqref{eqn:potential-flow-brain} to~\eqref{eqn:potential-flow-coupling}, such that
\begin{equation*} 
\mathfrak{K}^{-1}[\bm{q}^D, q^T, q^N] := [ (k^D)^{-1}\bm{q}^D, (k^T)^{-1}q^T, (k^N)^{-1}q^N ],
\end{equation*}
where $k^N = \operatorname{diag}(k^N_{e(i,j)})$. It is now straight forward to verify that with these definitions, the conservation laws \eqref{eqn:conservation-brain}-\eqref{eqn:conservation-term-node} can be summarized as 
\begin{equation}\label{eqn:conservation-md}
    \mathfrak{D}\cdot\mathfrak{q} = \mathfrak{r},
\end{equation}
where $\mathfrak{r}\equiv[r^D, r^N]$. Furthermore, the constitutive laws \eqref{eqn:potential-flow-brain} to~\eqref{eqn:potential-flow-coupling} can be summarized as 
\begin{equation}\label{eqn:potential-flow-md}
    \mathfrak{q} = -\mathfrak{K}\mathfrak{D}\mathfrak{p}.
\end{equation}

While the physical model formulation is satisfactory for non-degenerate $k^T_i(\bm{x})$, $i \in \mathcal{N}_T$, it will be beneficial to rescale the coupling flux to avoid considering a degenerate inner product when $k^T_i(\bm{x})\rightarrow 0$ for some points or region in $B_i$. To that aim, we introduce the square-root of the transfer coefficient~$k_i^S(\bm{x}) = \sqrt{k_i^T(\bm{x})}$, $i \in \mathcal{N}_T$, and the scaled transfer mass flux~$q_i^S(\bm{x}): \Omega \mapsto \mathbb{R}$, $i \in \mathcal{N}_T$, is defined as~$q_i^S(\bm{x}) = \left(k_i^S(\bm{x})\right)^{-1} q_i^T(\bm{x})$. Thus, we 
replace~\eqref{eqn:conservation-brain} \eqref{eqn:conservation-term-node}, and~\eqref{eqn:potential-flow-coupling} with
\begin{align}
&\text{(Conservation of mass in brain tissue)} \quad \nabla \cdot \bm{q}^D - \sum_{i\in \mathcal{N}_T} k_i^S q_i^S = r^D, \quad \text{in} \ \Omega, \label{eqn:conservation-brain-scale}	\\
&\text{(Conservation of mass at terminal nodes)}\int_{B_i} k_i^S q_i^S(\bm{x}) \mathrm{d} \bm{x} - q^N_{\mathcal{N}_i,i} = r_i^N \quad \text{for all} \ i \in \mathcal{N}_T, \label{eqn:conservation-term-node-scale} \\
&\text{(Potential flow from network to brain)} \quad q_i^S(\bm{x}) = -k_i^S(\bm{x}) \left( p^D - p_i^N \right), \quad \text{for} \ i  \in \mathcal{N}_T, \label{eqn:potential-flow-coupling-scale}
&
\end{align}
respectively. In this setting, we allow for degeneracy of the coupling term in the sense that we allow~$k_i^S(\bm{x}) \mapsto 0$. However, we require that $k_i^S$ is bounded from above, i.e., $k_i^S(\bm{x}) \leq C_{k^S}$ for $i \in \mathcal{N}_T$ and $\bm{x} \in \Omega$.  Furthermore, for all~$i$, we require it to hold that 
\begin{equation*}
\int_{B_i} k_i^S \mathrm{d} \bm{x} = c_{k^S_i} \geq c_{k^S} > 0,
\end{equation*}
where $c_{k^S}$ is a generic constant.  We note that a similar scaling has been applied previously to handle degeneracies occurring in mantle dynamics \cite{arbogast2017mixed} and flows in fractured porous media \cite{boon2018robust}. 

Equivalently, We denote the scaled mixed-dimensional flux on $\mathfrak{B}$ as $\mathfrak{q}^S \equiv [\bm{q}^D, q^S, q^N]$, and the scaling $\mathfrak{S}$ such that 
\begin{equation*} 
\mathfrak{S}^{-1}[\bm{q}^D, q^T, q^N] := [ \bm{q}^D, (k^S)^{-1} q^T, q^N ].
\end{equation*}
Thus, $\mathfrak{q}^S = \mathfrak{S}\mathfrak{q}$, and we can introduce the rescaled divergence and gradients as $\mathfrak{D}^S\cdot := \mathfrak{D}\cdot \mathfrak{S}$ and $\mathfrak{D}^S := \mathfrak{S}\mathfrak{D}$, respectively. The rescaled conservation equations are then summarized as
\begin{equation}\label{eqn:conservation-scaled-md}
    \mathfrak{D}^S\cdot\mathfrak{q}^S = \mathfrak{r}.
\end{equation}
The rescaled conservation equations are summarized as 
\begin{equation}\label{eqn:potential-flow-scaled-md}
    \mathfrak{q}^S = -\mathfrak{K}^S\mathfrak{D}^S\mathfrak{p}, 
\end{equation}
where $\mathfrak{K}^S = \mathfrak{S}^{-1}\mathfrak{K}\mathfrak{S}^{-1}$, and thus 
\begin{equation} \label{eqn:kS}
(\mathfrak{K}^S)^{-1}[\bm{q}^D, q^S, q^N] := [ (k^D)^{-1}\bm{q}^D, q^S, (k^N)^{-1}q^N ].
\end{equation}
Note in particular that $(\mathfrak{K}^S)^{-1}$ applied to $q^S$ has unit weight, and therefore does not degenerate.

\subsection{Weak Form}\label{subsec:weak-form}
In this subsection, we derive the weak formulation of the system. The development will be equally valid for both the original model, equations \eqref{eqn:conservation-md} and \eqref{eqn:potential-flow-md}, as well as the re-scaled model, equations \eqref{eqn:conservation-scaled-md} and \eqref{eqn:potential-flow-scaled-md}. Thus we will omit the superscript $S$ on the mixed-dimensional operators and variables to reduce notational overload. Nevertheless, in order to allow for degeneracies, we will always have the rescaled equations in mind, and thus when we need to specifically refer to $q^S$, and consider the coefficient $k^S$ to appear in the differential operator as opposed to the material law.

We first introduce proper function spaces on $\mathfrak{B}$.  We begin by defining a mixed-dimensional square-integrable space for pressure as follows,
\begin{equation*}
L^2(\mathfrak{B}) := L^2(\Omega) \times l^2(\mathcal{N}\backslash \mathcal{N}_D),
\end{equation*}
where $L^2(\Omega)$ is the standard $L^2$ space defined on domain $\Omega$ and $l^2(\mathcal{N}\backslash \mathcal{N}_D)$ is the standard $l^2$ space defined on the node set $\mathcal{N}\backslash \mathcal{N}_D$. For flux, we consider a space with bounded mixed-dimensional divergence as follows,
\begin{equation*}
H(\ddiv, \mathfrak{B}) := H(\ddiv, \Omega) \times \prod_{i\in \mathcal{N}_T}L^2(B_i) \times l^2(\mathcal{E}) 
\end{equation*} 
where $H(\ddiv, \Omega)$ is the space defined on $\Omega$ such that the functions and their divergence are both square-integrable. 
In addition, $L^2(B_i)$ are standard $L^2$ space defined on $B_i$, $i \in \mathcal{N}_T$, and $l^2(\mathcal{E})$ is the standard $l^2$ space defined on the edge set $\mathcal{E}$. 

We associate the mixed-dimensional space $L^2(\mathfrak{B})$ with the following inner product,
\begin{equation*}
(\mathfrak{p},\mathfrak{w}) = 
([p^D,p^N],[w^D,w^N]) := \int_{\Omega} p^D w^D \, \mathrm{d} \bm{x} + \sum_{i \in \mathcal{N}\backslash \mathcal{N}_D} p^N_i w^N_i, \quad \forall \, \mathfrak{p},\mathfrak{w} \in L^2(\mathfrak{B}). 
\end{equation*}
Similarly, we introduce the following inner product on~$H(\ddiv, \mathfrak{B})$, 
\begin{equation*}
(\mathfrak{q},\mathfrak{v}) = ([\bm{q}^D, q^S, q^N], [\bm{v}^D, v^S, v^N]) := \int_{\Omega} \bm{q}^D \cdot \bm{v}^D \, \mathrm{d}\bm{x} + \sum_{i\in \mathcal{N}_T} \int_{B_i} q_i^S v_i^S \, \mathrm{d}\bm{x} + \sum_{e(i,j) \in \mathcal{E}} q_{i,j}^N v_{i,j}^N.
\end{equation*}

It is important to note that the inner products are defined such that integration-by-parts holds in for the mixed-dimensional operators (both original and re-scaled cases).

\begin{lemma}[Integration by parts] \label{lem:IBP}
For any $\mathfrak{q} \in H(\ddiv, \mathfrak{B})$ and $\mathfrak{p} \in L^2(\mathfrak{B})$, we have
\begin{equation}\label{IBP-md}
    (\mathfrak{D}\mathfrak{p},\mathfrak{q}) + (\mathfrak{p},\mathfrak{D}\cdot \mathfrak{q}) = \int_{\partial\Omega} p^D\bm{q}^D\cdot\bm{n} \, \mathrm{d}\bm{x}
    +\sum_{i \in \mathcal{N}_D} p^N_i q_{\mathcal{N}_i,i}^N
\end{equation}
\end{lemma}
\begin{proof}
By a direct calculation (using the re-scaled operators and variables, the derivation for the original case is the same),  we have that 
\begin{align*}
(\mathfrak{p},\mathfrak{D}\cdot \mathfrak{q})
 =& \int_{\Omega} 
 p^D \left(\nabla \cdot \bm{q}^D - \sum_{i\in \mathcal{N}_T} k_i^S q_i^S\right) \, \mathrm{d} \bm{x} 
 -  \sum_{i \in \mathcal{N}_I \cup \mathcal{N}_N} p^N_i \sum_{j \in \mathcal{N}_i} q_{j,i}^N \\
& + \sum_{i \in \mathcal{N}_T} p^N_i \left(\int_{B_i} k_i^S q_i^S \, \mathrm{d} \bm{x} - q_{\mathcal{N}_i,i}^N\right)\\
 =& -\int_{\Omega} 
 \nabla p^D \cdot \bm{q}^D \, \mathrm{d} \bm{x} 
 + \int_{\partial\Omega} 
 p^D \bm{q}^D\cdot\bm{n}  \, \mathrm{d} \bm{x} 
 -\sum_{e(i,j) \in \mathcal{E}} (p^N_i-p^N_j) q_{i,j}^N\\
 &+\sum_{i \in \mathcal{N}_D} p^N_i q_{N_i,i}^N
 - \sum_{i \in \mathcal{N}_T}\int_{B_i} k_i^S(p^D- p^N_i)  q_i^S \, \mathrm{d} \bm{x}\\
 =& \int_{\partial\Omega} 
 p^D \bm{q}^D\cdot\bm{n}  \, \mathrm{d} \bm{x} 
 +\sum_{i \in \mathcal{N}_D} p^N_i q_{N_i,i}^N
 -(\mathfrak{D}\mathfrak{p},\mathfrak{q}),
 \end{align*}
 which completes the proof.
\end{proof}

To derive the weak formulation, we need to incorporate the boundary conditions.  Recall that we consider $\bm{q}^D \cdot \bm{n} = 0$ on $\partial \Omega$, therefore, we define the following functions space with boundary conditions,
\begin{equation*}
H_0(\ddiv, \mathfrak{B}) := H_0(\ddiv, \Omega) \times \prod_{i\in \mathcal{N}_T}L^2(B_i) \times l^2(\mathcal{E})  \subset H(\ddiv,\mathfrak{B}),
\end{equation*}
where $H_0(\ddiv, \Omega):= \{ \bm{q}^D \in H(\ddiv, \Omega) \,| \, \bm{q}^D \cdot \bm{n} = 0, \ \text{on} \ \partial \Omega \}$. In addition, with any material function $\mathfrak{K}$, we introduce a weighted inner product on $H(\ddiv, \mathfrak{B})$ as follows,
\begin{equation*}
(\mathfrak{q},\mathfrak{v})_{\mathfrak{K}^{-1}} :=
(\mathfrak{K}^{-1}\mathfrak{q},\mathfrak{v}).
\end{equation*}
Using the above function spaces and notation, together with the mixed-dimensional integration by parts formula~\eqref{IBP-md} and the homogeneous Dirichlet boundary conditions~\eqref{eqn:BC} on $\mathcal{N}_D$, i.e.,~$p_i^N = 0$, $i\in\mathcal{N}_D$, we have the following weak form for the conservation laws~\eqref{eqn:conservation-scaled-md}, and constitutive laws~\eqref{eqn:potential-flow-scaled-md}: Find~$\mathfrak{q}\in H_0(\ddiv, \mathfrak{B})$ and $\mathfrak{p}\in L^2(\mathfrak{B})$, such that
\begin{align}
&(\mathfrak{q},\mathfrak{v})_{\mathfrak{K}^{-1}} - (\mathfrak{p}, \mathfrak{D}\cdot\mathfrak{v})  = 0, \quad \forall \ \mathfrak{v} \in H_0(\ddiv, \mathfrak{B}), \label{eqn:weak-form-1}  \\
&-(\mathfrak{D} \cdot \mathfrak{q}, \mathfrak{w})  = -(\mathfrak{r}, \mathfrak{w}), \quad \forall \ \mathfrak{w} \in L^2(\mathfrak{B}). \label{eqn:weak-form-2}
\end{align}
Note that due to the integration by parts formula, if non-homogeneous boundary data is considered, this would appear as extra right-hand side terms in equation~\eqref{eqn:weak-form-1}. 

\subsection{Well-posedness} \label{sec:weak-wellposedness}
In this subsection, we focus on the well-posedness of the weak formulation~\eqref{eqn:weak-form-1}-\eqref{eqn:weak-form-2}. As in the previous subsection, it is understood that we are considering the re-scaled formulation, even though the superscript $S$ is suppressed. We first introduce the following norm on $L^2(\mathfrak{B})$,
\begin{equation} \label{eqn:norm-L2}
\| \mathfrak{p} \|_{L^2(\mathfrak{B})}^2 := (\mathfrak{p},\mathfrak{p}).
\end{equation}
And the following norm on $H(\ddiv, \mathfrak{B})$,
\begin{equation} \label{eqn:norm-Hdiv}
\| \mathfrak{q} \|_{H(\ddiv, \mathfrak{B})}^2 := \| \mathfrak{q} \|_{\mathfrak{K}^{-1}}^2 + \| \mathfrak{D}\cdot \mathfrak{q} \|_{L^2(\mathfrak{B})}^2,
\end{equation}
where 
\begin{equation} \label{eqn:weighted-norm}
\| \mathfrak{q} \|_{\mathfrak{K}^{-1}}^2 := ( \mathfrak{q},  \mathfrak{q})_{\mathfrak{K}^{-1}}.
\end{equation}
We emphasize that the weights in this norm do not degenerate for the re-scaled equations sine the unite weight is applied to~$q^S$, see equation~\eqref{eqn:kS}. 

The next lemma shows that the bilinear forms in the weak formulation~\eqref{eqn:weak-form-1}-\eqref{eqn:weak-form-2} are continuous.

\begin{lemma}[Continuity of~\eqref{eqn:weak-form-1}-\eqref{eqn:weak-form-2}] \label{lem:continuity}
For any $\mathfrak{q}, \mathfrak{v} \in H(\ddiv, \mathfrak{B})$ and $\mathfrak{w} \in L^2(\mathfrak{B})$, we have
\begin{align*}
(\mathfrak{q}, \mathfrak{v})_{\mathfrak{K}^{-1}} & \leq \| \mathfrak{q} \|_{H(\ddiv,\mathfrak{B})} \| \mathfrak{v} \|_{H(\ddiv,\mathfrak{B})}, \\
(\mathfrak{D} \cdot \mathfrak{q}, \mathfrak{w}) & \leq  \| \mathfrak{q} \|_{H(\ddiv,\mathfrak{B})} \| \mathfrak{w} \|_{L^2(\mathfrak{B})}.
\end{align*}
\end{lemma}
\begin{proof}
The continuity of both bilinear forms follow directly from the Cauchy-Schwarz inequality and the definition of the norms~\eqref{eqn:norm-Hdiv} and~\eqref{eqn:norm-L2}.
\end{proof}

Now we show the ellipticity of the inner product $(\cdot, \cdot)_{\mathfrak{K}^{-1}}$ on the kernel of the mixed-dimensional divergence operator $\mathfrak{D} \cdot$ in the following lemma.
\begin{lemma}[Ellipticity of~\eqref{eqn:weak-form-1}-\eqref{eqn:weak-form-2}] \label{lem:ellipticity}
If $\mathfrak{q} \in H(\ddiv, \mathfrak{B})$ satisfies
\begin{equation} \label{eqn:null-b}
(\mathfrak{D}\cdot \mathfrak{q}, \mathfrak{w}) = 0, \quad \forall \, \mathfrak{w} \in L^2(\mathfrak{B}),
\end{equation} 
then
\begin{equation} \label{eqn:ellipticity}
(\mathfrak{q}, \mathfrak{q})_{\mathfrak{K}^{-1}} = \| \mathfrak{q} \|_{H(\ddiv, \mathfrak{B})}^2.
\end{equation}
\end{lemma}
\begin{proof}
Since $\mathfrak{D}\cdot \mathfrak{q} \in L^2(\mathfrak{B})$, from~\eqref{eqn:null-b}, we have 
\begin{equation*}
\| \mathfrak{D}\cdot \mathfrak{q} \|_{L^2(\mathfrak{B})} = 0.
\end{equation*}
Therefore, \eqref{eqn:ellipticity} follows directly from the above identity and the definition of the norm~\eqref{eqn:norm-Hdiv}.
\end{proof}

Next, we discuss the inf-sup condition of the bilinear form $(\mathfrak{r}, \mathfrak{D} \cdot \mathfrak{q})$ in the following lemma.
\begin{lemma}[Inf-sup condition of~\eqref{eqn:weak-form-1}-\eqref{eqn:weak-form-2}] \label{lem:inf-sup}
	There exists a constant $\beta > 0$ such that, for any given function $\mathfrak{r}\in L^2(\mathfrak{B})$,
	\begin{equation}\label{ine:inf-sup}
	\sup_{\mathfrak{q}\in H_0(\ddiv,\mathfrak{B})} \frac{(\mathfrak{r}, \mathfrak{D} \cdot \mathfrak{q})}{ \| \mathfrak{q} \|_{H(\ddiv,\mathfrak{B})} } \geq \beta \| \mathfrak{r} \|_{L^2(\mathfrak{B})}.
	\end{equation}
	Here, the inf-sup constant $\beta$ depends on  $|B_i| = \operatorname{measure}(B_i)$, the maximal number of overlaps between $B_i$, structure of the trees $\mathcal{T} \in \mathcal{F}$, the domain $\Omega$, and the constants $c_{k^S}$ and $C_{k^S}$. 
\end{lemma}

\begin{proof}
Assume $\mathfrak{r} = [r^D, r^N] \in L^2(\mathfrak{B})$ given, we first aim to construct $\mathfrak{q} = [\bm{q}^D, q^S, q^N] \in H_0(\ddiv,\mathfrak{B})$ such that $\mathfrak{D} \cdot [\bm{q}^D, q^S, q^N] = [r^D, r^N]$.  

First step is to construct $q^N$ based on the forest~$\mathcal{F}$.  Based on the signed incidence matrix $\mathcal{G}$~\eqref{eqn:incidence-matrix-forest}, we omit those columns that correspond to the Dirichlet root nodes to obtain the signed incidence matrix with boundary conditions $\mathcal{G}_{\mathcal{F}}$.  Then, we consider the following mixed-formulation graph Laplacian problem: Find $q_{\mathcal{F}} \in \mathbb{R}^{|\mathcal{E}|}$ and $\psi_{\mathcal{F}} \in \mathbb{R}^{|\mathcal{N}|-|\mathcal{N}_D|}$
\begin{eqnarray} 
\mathfrak{K}^{-1} q_{\mathcal{F}} - \mathcal{G}_{\mathcal{F}} \psi_{\mathcal{F}}  &= 0, \label{eqn:mixed-graph-laplacian-flux} \\
\mathcal{G}^T_{\mathcal{F}} q_{\mathcal{F}} &= r_{\mathcal{F}}. \label{eqn:mixed-graph-laplacian-pressure}
\end{eqnarray}
Here, for trees $\mathcal{T} \in \mathcal{F}_N$, we set $(r_{\mathcal{F}})_i = r_i^N$, $i \in \mathcal{N}_{\mathcal{T}, N} \cup \mathcal{N}_{\mathcal{T},I}$, and for $i \in \mathcal{N}_{\mathcal{T}, T}$, we choose $(r_{\mathcal{F}})_i$ such that $\sum_{i \in \mathcal{N}_{\mathcal{T}}} (r_{\mathcal{F}})_i = 0$.  The choice is not unique, and here we choose
\begin{equation} \label{eqn:r_i-on-terminal-node}
(r_{\mathcal{F}})_i = r_i^N - \frac{\sum_{i \in \mathcal{N}_{\mathcal{T}}} r_i^N }{|\mathcal{N}_{\mathcal{T},T}|}, \quad i \in \mathcal{N}_{\mathcal{T},T}, \quad \mathcal{T} \in \mathcal{F}_N.
\end{equation}
For trees $\mathcal{T} \in \mathcal{F}_D$, we set $(r_{\mathcal{F}})_i = r_i^N$, $i \in \mathcal{N}_{\mathcal{T},T} \cup \mathcal{N}_{\mathcal{T},I}$, and, for $i \in \mathcal{N}_{\mathcal{T},T}$, we set
\begin{equation*}
(r_{\mathcal{F}})_i = r_i^N + \frac{1}{| \mathcal{N}_{\mathcal{F}_D,T} |} \int_{\Omega} r^D \ \mathrm{d} \bm{x} + \frac{1}{|\mathcal{N}_{\mathcal{F}_D, T}|} \sum_{i \in \mathcal{N}_{\mathcal{F}_N}}r_i^N, \quad i \in \mathcal{N}_{\mathcal{F}_D,T}.
\end{equation*}
The reason of such a choice will be made clear later in the proof when we construct $\bm{q}^D$. Note that, since the degree of node $i \in \mathcal{N}_{\mathcal{T},T}$ is one, once $(r_{\mathcal{F}})_i$ is fixed, we natrually have $(q_{\mathcal{F}})_{e(\mathcal{N}_i,i)} = - (r_{\mathcal{F}})_i$.  With this choice of $r_{\mathcal{F}}$, the mixed-formulation graph Laplacian problem~\eqref{eqn:mixed-graph-laplacian-flux}-\eqref{eqn:mixed-graph-laplacian-pressure} is well-posed in the sense that $\psi_{\mathcal{F}}$ is unique (up to a constant on the trees $\mathcal{T} \in \mathcal{F}_N$) and $q_{\mathcal{F}}$ is uniquely defined.  Once $q_{\mathcal{F}}$ is obtained, we define $q^N$ by $q^N_{i,j} = (q_{\mathcal{F}})_{e(i,j)}$, $e(i,j) \in \mathcal{E}$. 

From the mixed-formulation~\eqref{eqn:mixed-graph-laplacian-flux}-\eqref{eqn:mixed-graph-laplacian-pressure}, we have the following estimates,
\begin{equation}\label{ine:estimates-graph}
\| \mathcal{G}_{\mathcal{F}}^T q_{\mathcal{F}} \|^2 = \| r_{\mathcal{F}} \|^2 \quad \text{and} \quad (\mathfrak{K}^{-1}q_{\mathcal{F}}, q_{\mathcal{F}}) \leq (\lambda_{\min}^{\mathcal{F}})^{-1} \| r_{\mathcal{F}} \|^2,
\end{equation}
where $\lambda_{min}^{\mathcal{F}}$ is the smallest non-zero eigenvalue of the weighted graph Laplaican of the forest $\mathcal{F}$, i.e, $\mathcal{L}_{\mathcal{F}} = \mathcal{G}_{\mathcal{F}}^T \mathfrak{K} \mathcal{G}_{\mathcal{F}}$.  We comment that $\lambda_{\min}^{\mathcal{F}}$ is bounded below by the so-called Cheeger constant of the graph, so depends on the structure of the trees $\mathcal{T}$ in the forest $\mathcal{F}$. Note that 
\begin{equation*}
\| r_{\mathcal{F}} \|^2 = \sum_{i \in \mathcal{N}_I \cup \mathcal{N}_N} (r_i^N)^2 + \sum_{i \in \mathcal{N}_{\mathcal{F}_D,T}} ((r_{\mathcal{F}})_i)^2 + \sum_{i \in \mathcal{N}_{\mathcal{F}_N,T}} ((r_{\mathcal{F}})_i)^2
\end{equation*}
and, due the choice~\eqref{eqn:r_i-on-terminal-node}, the last term on the right-hand-side can be bounded by
\begin{equation}\label{ine:estimates-terminal-node-N}
\sum_{i \in \mathcal{N}_{\mathcal{F}_N,T}} ((r_{\mathcal{F}})_i)^2 \leq C_N \sum_{i \in \mathcal{N}_{\mathcal{F}_N}} (r_i^N)^2
\end{equation}
with $C_N = 2 \left(\max_{\mathcal{T} \in \mathcal{F}_N} \frac{|\mathcal{N}_{\mathcal{T}}|}{|\mathcal{N}_{\mathcal{T},T}|} + 1 \right)$. Similarly, by Cauchy-Schwarz inequality, the second term on the right-hand-side can be bounded as follows,
\begin{equation} \label{ine:estimates-terminal-node-D}
\sum_{i \in \mathcal{N}_{\mathcal{F}_D,T}} ((r_{\mathcal{F}})_i)^2 \leq C_D \left[ \sum_{i \in \mathcal{N}_{\mathcal{F}_D,T}} (r_i^N)^2 + \sum_{i \in \mathcal{N}_{\mathcal{F}_N}} (r_i^N)^2 + \int_{\Omega} (r^D)^2 \, \mathrm{d} \bm{x} \right],
\end{equation}
where $C_D = 3 \max \Big\{ 1, \frac{|\Omega|}{|\mathcal{N}_{\mathcal{F}_D,T}|}, \frac{|\mathcal{N}_{\mathcal{F}_N}|}{|\mathcal{N}_{\mathcal{F}_D, T}|}  \Big\}$.

 Therefore, combining the estimates~\eqref{ine:estimates-graph}, \eqref{ine:estimates-terminal-node-N}, \eqref{ine:estimates-terminal-node-D}, and the definitions of $\mathcal{G}_{\mathcal{F}}$ and $q^N$, the following estimate holds,
 \begin{align} 
 \quad \sum_{e(i,j) \in \mathcal{E}} (k^N_{e(i,j)})^{-1} |q_{i,j}^N|^2 \leq C_{q^N} \left[  \sum_{i \in \mathcal{N}_I \cup \mathcal{N}_N}(r_i^N)^2  + \sum_{i \in \mathcal{N}_{T}} (r_i^N)^2  + \int_{\Omega} (r^D)^2 \, \mathrm{d} \bm{x}  \right],  \label{ine:estimate-qN}
 \end{align}  
 where $C_{q^N} = 
 (\lambda_{\min}^{\mathcal{F}})^{-1}(C_N+C_D+1)$.

Next we construct $q^S$ from $q^N$ and $r^N$ so that~\eqref{eqn:conservation-term-node-scale} is satisfied exactly, i.e., we define, for each terminal nodes $i \in \mathcal{N}_T$, 
\begin{equation} \label{eqn:qS}
q_i^S(\bm{x}) = \frac{q_{\mathcal{N}_i,i}^N + r_i^N}{ c_{k_i}}, \qquad \bm{x} \in B_i.
\end{equation}
From the construction, we have
\begin{align} 
\sum_{i \in \mathcal{N}_T} \int_{B_i} |q_i^S(\bm{x})|^2 \, \mathrm{d}\bm{x}  &\leq \frac{2 |B_i|}{c_{k^S}^2} \left[  \sum_{i\in \mathcal{N}_T} |r_i^N|^2 +\sum_{i \in \mathcal{N}_T} |q^N_{\mathcal{N}_i,i}|^2  \right] \nonumber \\
& \leq C^1_{q^S} \left[  \sum_{i\in \mathcal{N}_T} |r_i^N|^2 +\sum_{i \in \mathcal{N}_{\mathcal{F}_N}} |r_i^N|^2  + \int_{\Omega} (r^D)^2 \, \mathrm{d} \bm{x} \right],  \label{ine:estimate-qS-1}
\end{align}
where $C_{q^S}^1 = \frac{2|B_i|}{c_{k^S}^2} (C_N + C_D)$. Here we use the fact that $q^N_{N_i,i} = - (r_{\mathcal{F}})_i$ for $i \in \mathcal{N}_{T}$ by our construction of $q^N$, and the estimates~\eqref{ine:estimates-terminal-node-N} and~\eqref{ine:estimates-terminal-node-D} in the last step. Similarly, we also have
\begin{equation} \label{ine:estimate-qS-2}
\sum_{i \in \mathcal{N}_T}\int_{B_i} |k_i^S q_i^S|^2 \, \mathrm{d}\bm{x} \leq C_{q^S}^2 \left[  \sum_{i\in \mathcal{N}_T} |r_i^N|^2 +\sum_{i \in \mathcal{N}_{\mathcal{F}_N}} |r_i^N|^2  + \int_{\Omega} (r^D)^2 \, \mathrm{d} \bm{x}  \right]
\end{equation}
with $C_{q^S}^2 = 2\frac{C_{k^S}}{c_{k^S}} (C_N+C_D+1) $.
 

Finally, we consider the following mixed-formulation Laplacian problem 
\begin{align}
(k^D)^{-1} \bm{q}^D + \nabla \psi &= 0 \label{eqn:mix-fem-qN-1} \\
\nabla \cdot \bm{q}^D  &= r^D + \sum_{i \in \mathcal{N}_T} k_i^S q_i^S  \label{eqn:mix-fem-qN-2}
\end{align}
with boundary condition $\bm{q}^D \cdot \bm{n} = 0$ on $\partial \Omega$. 
This problem is well-posed because 
\begin{align*}
&\quad \int_{\Omega} r^D \, \mathrm{d} \bm{x} + \sum_{ i \in \mathcal{N}_T} \int_{B_i} k_i^S(\bm{x})q_i^S(\bm{x}) \, \mathrm{d} \bm{x} \\
& = \int_{\Omega}r^D \, \mathrm{d} \bm{x} + \sum_{i \in \mathcal{N}_T} \int_{B_i} k_i^S(\bm{x})\frac{q_{\mathcal{N}_i,i}^N + r_i^N}{c_{k_i}} \, \mathrm{d} \bm{x} \\
& = \int_{\Omega} r^D \, \mathrm{d} \bm{x} + \sum_{i \in \mathcal{N}_T} \left( q_{\mathcal{N}_i,i}^N + r_i^N \right) \\
& = \int_{\Omega} r^D \, \mathrm{d} \bm{x} + \sum_{i \in \mathcal{N}_{\mathcal{F}_N,T}} \left( q_{\mathcal{N}_i,i}^N + r_i^N \right)  + \sum_{i \in \mathcal{N}_{\mathcal{F}_D,T}} \left( q_{\mathcal{N}_i,i}^N + r_i^N \right)  \\
& = \int_{\Omega} r^D \, \mathrm{d} \bm{x} + \sum_{\mathcal{T} \in \mathcal{F}_N} \sum_{i \in \mathcal{N_T}} \left( -r_i^N + \frac{\sum_{i \in \mathcal{N_T}}r_i^N}{|\mathcal{N}_{\mathcal{T},T}|} + r_i^N \right) \\
& \quad + 
\sum_{i \in \mathcal{N}_{\mathcal{F}_D,T}} \left(  -r_i^N - \frac{1}{| \mathcal{N}_{\mathcal{F}_D,T} |} \int_{\Omega} r^D \ \mathrm{d} \bm{x} - \frac{1}{|\mathcal{N}_{\mathcal{F}_D, T}|} \sum_{i \in \mathcal{N}_{\mathcal{F}_N}}r_i^N + r_i^N \right) \\
& = 0,
\end{align*}
which verifies the consistency of the data with respect to the pure Nuemann boundary condition $\bm{q}^D \cdot \bm{n} = 0$ on $\partial \Omega$. Furthermore, the following estimate holds,
\begin{align}
\int_{\Omega} |\nabla \cdot \bm{q}^D|^2 \,\mathrm{d} \bm{x} &= \int_{\Omega} | r^D + \sum_{i \in \mathcal{N}_T}k_i^S q_i^S|^2 \, \mathrm{d} \bm{x} \nonumber \\
& \leq 2 \int_{\Omega} |r^D|^2 \mathrm{d} \bm{x} + 2 N_{B_i} \sum_{i \in \mathcal{N}_T} \int_{B_i} |k_i^Sq_i^S|^2 \, \mathrm{d} \bm{x} \nonumber \\
& \leq C_{\bm{q}^D}^1 \left[ \int_{\Omega} |r^D|^2 \, \mathrm{d}\bm{x} + \sum_{i \in \mathcal{N}_T} |r_i^N|^2 + \sum_{i \in \mathcal{N}_{\mathcal{F}_N}} |r_i^N|^2  \right] \label{ine:estimate-qD-1}
\end{align}
where $C_{\bm{q}^D}^1 = 2 \left( N_{B_i} C_{q^S}^2  +1 \right)$ and $N_{B_i}$ is the maximal number of the overlapping between the $B_i$, $i \in \mathcal{N}_T$.  Similarly, we also have
\begin{align}
\int_{\Omega} (k^D)^{-1} |\bm{q}^D|^2 \, \mathrm{d} \bm{x} &\leq C_p^{-1} \int_{\Omega} |r^D + \sum_{i \in \mathcal{N}_T} k_i^S q_i^S |^2 \, \mathrm{d} \bm{x} \nonumber \\
& \leq C_{\bm{q}^D}^2  \left[ \int_{\Omega} |r^D|^2 \, \mathrm{d}\bm{x} + \sum_{i \in \mathcal{N}_T} |r_i^N|^2 + \sum_{i \in \mathcal{N}_{\mathcal{F}_N}} |r_i^N|^2  \right] \label{ine:estimate-qD-2}
\end{align}
where $C_{\bm{q}^D}^2 = C_p^{-1} C_{\bm{q}^D}^1$ and $C_p$ is the weighted Poincare constant, i.e, $C_p (v,v) \leq ((k^D) \nabla v, \nabla v)$.  

%

Now $[\bm{q}^D, q^S, q^N]$ has been constructed based on $[r^D, r^N]$ and it satisfies 
\begin{equation} \label{eqn:Dq=r}
\mathfrak{D} \cdot [\bm{q}^D, q^S, q^N] = [r^D, r^N],
\end{equation}
and we have
\begin{align*}
\| [\bm{q}^D, q^S, q^N] \|^2_{H(\ddiv,\mathfrak{B})} &= \| [\bm{q}^D, q^S, q^N] \|^2_{\mathfrak{K}^{-1}} + \| \mathfrak{D} \cdot [\bm{q}^D, q^S, q^N] \|^2_{\mathfrak{L^2}} \\
& = \int_{\Omega} (k^D)^{-1} |\bm{q}^D|^2 \, \mathrm{d} \bm{x} + \sum_{i \in \mathcal{N}_T} \int_{B_i} |q_i^S|^2 \, \mathrm{d}\bm{x} \\
& \quad + \sum_{e(i,j)\in \mathcal{E}} (k^N_{e(i,j)})^{-1} |q^N_{i,j}|^2 
 + \| [r^D,r^N ]\|_{L^2(\mathfrak{B})}.
\end{align*}
Now, based on~\eqref{ine:estimate-qN}, \eqref{ine:estimate-qS-1}, and \eqref{ine:estimate-qD-2},  we can derive that
\begin{equation}
\| [\bm{q}^D, q^S, q^N] \|^2_{H(\ddiv,\mathfrak{B})} \leq C_{\beta}  \| [r^D, r^N] \|^2_{L^2(\mathfrak{B})} 
\end{equation}
with $C_\beta = 2C_{\bm{q}^D}^2 + 2 C_{q^S}^1 + C_{q^N} + 1$.
Then the inf-sup condition~\eqref{ine:inf-sup} hold with $\beta = C_{\beta}^{-1}$.

\end{proof}

\begin{remark}\label{rem:scaling}
The inf-sup proof shows the importance of the using the scaled equations \eqref{eqn:conservation-scaled-md} and \eqref{eqn:potential-flow-scaled-md} in the case where $k^T$ goes to zero. Indeed, for the non-scaled equations, a similar approach would lead to an inf-sup constant depending on the pointwise lower bound on $\inf_{\bm{x}\in B_i}(k^T_i(\bm{x}))$, which may not be positive. In contrast, as seen in the proof above, for the scaled equations, inf-sup constant depends on the much less restrictive integrated bound $c_{k_i^S}$.
\end{remark}

We now have the following well-posedness results. 
\begin{theorem}[Well-posedness of~\eqref{eqn:weak-form-1}-\eqref{eqn:weak-form-2}] \label{thm:wellposedness}
The weak formulation~\eqref{eqn:weak-form-1} and~\eqref{eqn:weak-form-2} is well-posed with respect to the norms~\eqref{eqn:norm-Hdiv} and~\eqref{eqn:norm-L2}.
\end{theorem}
\begin{proof}
The result follows directly from the standard theory for saddle point problems, see, e.g.~\cite{boffi2013mixed}, and  Lemmas~\ref{lem:continuity}, \ref{lem:ellipticity}, and~\ref{lem:inf-sup}.
\end{proof}

\section{Finite-element Approximation} \label{sec:FEM}
In this section, we propose the finite-element approximation for solving the weak formulation~\eqref{eqn:weak-form-1}-\eqref{eqn:weak-form-2}. The coupling between the graph and the porous domain, as well as the heterogeneous nature of the parameters found in applications, suggests that it is natural to consider low-order approximations. As a consequence, we only consider the lowest-order approximation here, recognizing that higher-order spaces can be introduced in the mixed formulation.  

\subsection{Mixed Finite-Element Method}
Given a mesh~$\mathcal{M}$ of the domain $\Omega$, e.g., triangles/quadrilaterals in 2D and tetrahedrons/cuboids in 3D, we consider the standard RT0/P0 finite element for approximating the fluid flux $\bm{q}^D$ and pressure $P^D$ in the domain and denote them by $H_{h}(\ddiv, \mathcal{M})$ and $P_0(\mathcal{M})$, respectively.  For node pressure potentials $p^N$, we use vertex degrees of freedom (DOFs) of the graph.  For fluid flux on the tree edges, we use edge DOFs of the graph.  For the fluid flux transferring from terminal $i$ to point $\bm{x}$, it appears natural to consider the piece-wise constant finite element on $\mathcal{M}_i$ (denoted as $P_0(\mathcal{M}_i)$), which is the restriction of $\mathcal{M}$ to $B_i$, i.e. $\mathcal{M}_i=\mathcal{M}\cap B_i$.  In summary, we consider the following conforming finite-element spaces
\begin{equation*}
H_{h}(\ddiv,\mathfrak{B}) :=  H_{h}(\ddiv, \mathcal{M}) \times \prod_{i\in \mathcal{N}_T}P_0(\mathcal{M}_i) \times l^2(\mathcal{E})  \subset H(\ddiv,\mathfrak{B}),
\end{equation*}
its corresponding finite-element space with boundary conditions,
\begin{equation*}
H_{h,0}(\ddiv,\mathfrak{B}) :=  H_{h,0}(\ddiv, \mathcal{M}) \times \prod_{i\in \mathcal{N}_T}P_0(\mathcal{M}_i) \times l^2(\mathcal{E})  \subset H_0(\ddiv,\mathfrak{B}),
\end{equation*}
where $H_{h,0}(\ddiv, \mathcal{M}) := \{ \bm{q}_h^D \in H_h(\ddiv, \mathcal{M})\, | \, \bm{q}_h^D \cdot \bm{n} = 0, \ \text{on} \ \partial \Omega \}$, and 
\begin{equation*}
\mathfrak{L}_h^2 := P_0(\mathcal{M}) \times l^2(\mathcal{N}\backslash \mathcal{N}_D) \subset L^2(\mathfrak{B}).
\end{equation*}
Using the finite-element spaces introduced above, the mixed finite-element approximation of~\eqref{eqn:weak-form-1}-\eqref{eqn:weak-form-2} is: Find~$\mathfrak{q}_h := [\bm{q}_h^D, q_h^S, q_h^N]\in H_{h,0}(\ddiv,\mathfrak{B})$ and $\mathfrak{p}_h := [p_h^D, p_h^N]\in \mathfrak{L}_h^2$, such that
\begin{align}
&(\mathfrak{q}_h, \mathfrak{v}_h)_{\mathfrak{K}^{-1}} - (\mathfrak{p}_h, \mathfrak{D}\cdot \mathfrak{q}_h)  = 0, \quad \forall \ \mathfrak{v}_h \in H_{h,0}(\ddiv,\mathfrak{B}), \label{eqn:mixed-fem-1}  \\
&-(\mathfrak{D} \cdot \mathfrak{q}_h, \mathfrak{w}_h)  = -(\mathfrak{r}, \mathfrak{w}_h), \quad \forall \ \mathfrak{w}_h \in \mathfrak{L}_h^2. \label{eqn:mixed-fem-2}
\end{align}

\begin{remark}
By considering a test function $\mathfrak{w}_h$ which is constant on a $B_i$, we verify from equations  \ref{eqn:conservation-brain-scale} and \ref{eqn:conservation-term-node-scale} that the physical flux $q^T = k_i^S q_i^S$ is conserved. 
We note that the lowest-order mixed finite element approximation is locally conservative even when applied to scaled variables, in contrast to the situation observed when similar scalings are applied in the physical dimensions of $\Omega$ (see e.g. \cite{arbogast2017mixed}). 
\end{remark}

\subsection{Well-posedness} 
In this subsection, we consider the well-posedness of the mixed finite-element approximation~\eqref{eqn:mixed-fem-1}-\eqref{eqn:mixed-fem-2}. It is essentially the same as the well-posedness analysis for the weak formulation in Section~\ref{sec:weak-wellposedness}, and our presentation will therefore be brief. 

Since we use conforming finite-element spaces, the continuity results (Lemma~\ref{lem:continuity}) holds naturally on the discrete level. 
\begin{lemma}[Continuity of~\eqref{eqn:mixed-fem-1}-\eqref{eqn:mixed-fem-2}] \label{lem:fem-continuity}
	For any $\mathfrak{q}_h$, $\mathfrak{v}_h \in H_{h,0}(\ddiv,\mathfrak{B})$ and $\mathfrak{w}_h \in \mathfrak{L}_h^2$, we have
	\begin{align*}
	(\mathfrak{q}_h, \mathfrak{v}_h)_{\mathfrak{K}^{-1}} & \leq \| \mathfrak{q}_h \|_{H(\ddiv,\mathfrak{B})} \| \mathfrak{v}_h \|_{H(\ddiv,\mathfrak{B})}, \\
	(\mathfrak{D} \cdot \mathfrak{q}_h, \mathfrak{w}_h) & \leq   \| \mathfrak{q}_h \|_{H(\ddiv,\mathfrak{B})} \| \mathfrak{w}_h \|_{L^2(\mathfrak{B})}.
	\end{align*}
\end{lemma}

For the ellipticity (Lemma~\ref{lem:ellipticity}), using the fact that the finite dimensional spaces are conforming in the sense that for $\mathfrak{q}_h \in H_{h,0}(\ddiv,\mathfrak{B})$, then it holds that $\mathfrak{D}\cdot\mathfrak{q}_h \in \mathcal{L}^2_h$, then the continuous results hold on the discrete level. 
\begin{lemma}[Ellipticity of~\eqref{eqn:mixed-fem-1}-\eqref{eqn:mixed-fem-2}] \label{lem:fem-ellipticity}
	If $\mathfrak{q}_h \in H_h(\ddiv,\mathfrak{B})$ satisfies
	\begin{equation} \label{eqn:fem-null-b}
	(\mathfrak{D}\cdot \mathfrak{q}_h, \mathfrak{w}_h) = 0, \quad \forall \, \mathfrak{w}_h \in \mathfrak{L}_h^2,
	\end{equation} 
	then
	\begin{equation} \label{eqn:fem-ellipticity}
	(\mathfrak{q}_h, \mathfrak{q}_h)_{\mathfrak{K}^{-1}} = \| \mathfrak{q}_h \|_{H(\ddiv,\mathfrak{B})}^2
	\end{equation}
\end{lemma}

Moreover, the inf-sup condition (Lemma~\ref{lem:inf-sup}) can be derived in a similar fashion on the discrete level as well. 
\begin{lemma}[Inf-sup condition of~\eqref{eqn:mixed-fem-1}-\eqref{eqn:mixed-fem-2}] \label{lem:fem-inf-sup}
	There exists a constant $\beta>0$ such that, for any given function $\mathfrak{r}_h\in L^2_h(\mathfrak{B})$,
	\begin{equation}\label{ine:fem-inf-sup}
	\sup_{\mathfrak{q}_h\in H_{h,0}(\ddiv,\mathfrak{B})} \frac{(\mathfrak{r}_h, \mathfrak{D} \cdot \mathfrak{q}_h)}{ \| \mathfrak{q}_h \|_{H(\ddiv,\mathfrak{B})} } \geq \beta \| \mathfrak{r}_h \|_{L^2(\mathfrak{B})}.
	\end{equation}
	Here, the inf-sup constant $\beta$ depends on $|M_i| = \operatorname{measure}(\mathcal{M}_i) = \mathcal{O}(h^n)$, the maximal number of overlaps between $B_i$, structure of the trees $\mathcal{T} \in \mathcal{F}$, the domain $\Omega$, and the constants $c_{k^S}$ and $C_{k^S}$. 
\end{lemma}
\begin{proof}
Given $[r^D_h, r^N_h] \in L^2_h(\mathfrak{B})$, the construction of $[\bm{q}_h^D, q_h^S, q_h^N] \in H_{h,0}(\ddiv,\mathfrak{B})$ is similar to the construction presented in the proof of Lemma~\ref{lem:inf-sup}. $q_h^N$ can be constructed exactly the same as the construction of $q^N$.  Then $q_h^S$ can be defined as~\eqref{eqn:qS} as well since such construction also makes sure that $q_h^S \in \prod_{i\in \mathcal{N}_T} P_0(\mathcal{M}_i)$.  The construction of $\bm{q}_h^D$ should be obtained by solving~\eqref{eqn:mix-fem-qN-1}-\eqref{eqn:mix-fem-qN-2} with a mixed finite-element method using $H_{h,0}(\ddiv, \mathcal{M})$ and $P_0(\mathcal{M})$.  Such construction also makes sure that
\begin{equation*}
\mathfrak{D} \cdot [\bm{q}_h^D, q_h^S, q_h^N] = [r_h^D, r_h^N],
\end{equation*} 
and
\begin{equation*}
\| [\bm{q}^D_h, q^S_h, q^N_h] \|_{H(\ddiv,\mathfrak{B})} \leq C_{\beta} \|[r^D_h, r^N_h]\|_{\mathfrak{L^2}}.
\end{equation*}
Therefore the inf-sup condition~\eqref{ine:fem-inf-sup} follows directly. 
\end{proof}

Thus, the well-posedness of the mixed finite-element approximation~\eqref{eqn:mixed-fem-1}-\eqref{eqn:mixed-fem-2} follows from Lemmas~\ref{lem:fem-continuity}, \ref{lem:fem-ellipticity}, and~\ref{lem:fem-inf-sup}.
\begin{theorem}[Well-posedness of~\eqref{eqn:mixed-fem-1}-\eqref{eqn:mixed-fem-2}] \label{thm:fem-wellposedness}
	The weak formulation~\eqref{eqn:mixed-fem-1} and~\eqref{eqn:mixed-fem-2} is well-posed with respect to the norms~\eqref{eqn:norm-Hdiv} and~\eqref{eqn:norm-L2}.
\end{theorem}

\subsection{Convergence}
Based on Lemma~\ref{lem:fem-continuity}, \ref{lem:fem-ellipticity}, and~\ref{lem:fem-inf-sup} and applying the general theory of Galerkin methods, see~\cite{brezzi2012mixed,boffi2013mixed}, we immediately gives a quasi-optimality error estimate.
\begin{theorem}
Suppose that $\mathfrak{q} \in H_0(\ddiv,\mathfrak{B})$ and $\mathfrak{p}\in L^2(\mathfrak{B})$ satisfy the weak formulation~\eqref{eqn:weak-form-1}-\eqref{eqn:weak-form-2}, then the finite-element solution $\mathfrak{q}_{h} \in H_{h,0}(\ddiv,\mathfrak{B})$ and $\mathfrak{p}_h\in \mathfrak{L}_h^2$ of the mixed fintie-element approximation~\eqref{eqn:mixed-fem-1}-\eqref{eqn:mixed-fem-2} satisfy that
\begin{align}
& \quad  \| \mathfrak{q} - \mathfrak{q}_h\|_{H(\ddiv,\mathfrak{B})} + \| \mathfrak{p} - \mathfrak{p}_h \|_{L^2(\mathfrak{B})} \nonumber \\ 
& \leq c \left(  \inf_{\mathfrak{v}_h \in H_{h,0}(\ddiv,\mathfrak{B})} \| \mathfrak{q} - \mathfrak{v}_h\|_{H(\ddiv,\mathfrak{B})} + \inf_{\mathfrak{w}_h \in \mathfrak{L}_h^2 } \| \mathfrak{p} - \mathfrak{w}_h \|_{L^2(\mathfrak{B})}   \right), \label{ine:fem-error}
\end{align}
where the constant $c$ depends on $\beta$.
\end{theorem}

As usual, to obtain the final convergence result, we use interpolations to bound the right-hand-side of the above error estimate~\eqref{ine:fem-error}.  Here, we choose $\bm{v}_h^D = \pi_{\ddiv}\bm{q}^D$, where $\pi_{\ddiv}: H^1(\Omega) \mapsto H_{h}(\ddiv, \mathcal{M})$ is the standard interpolation given by the $H_{h}(\ddiv, \mathcal{M})$ degrees of freedom, $v_h^S = \pi_0 q^S$, where $\pi_0$ denotes the standard piecewice constant interpolation, and $v_h^N = q^N$.  With those choices and the classical error estimates for interpolations, together with Cauchy-Schwarz inequality, we naturally have
\begin{equation*}
\| [\bm{q}^D, q^S, q^N] - [\pi_{\ddiv}\bm{q}^D, \pi_0 q^S, q^N] \|_{H(\ddiv,\mathfrak{B})} \leq c h \left(  \| \bm{q}^D \|_1^2 + \| \nabla \cdot \bm{q}^D \|_1^2 + \sum_{i \in \mathcal{N}_T} \| q_i^S \|^2_1  \right)^{\frac{1}{2}}.
\end{equation*}
Similarly, by choosing $w_h^D = \pi_0 p^D$ and $w_h^N = p^N$, we have
\begin{equation*}
\| [p^D, p^N]  - [\pi_0 p^D, p^N]\|_{L^2(\mathfrak{B})} \leq ch \| p^D \|_1.
\end{equation*}
Therefore, we have the overall convergence result for the finite-element method~\eqref{eqn:mixed-fem-1}-\eqref{eqn:mixed-fem-2} as follows.
\begin{corollary} \label{coro:fem-convergence}
Suppose that $\mathfrak{q} \in H_0(\ddiv,\mathfrak{B})$ and $\mathfrak{p}\in L^2(\mathfrak{B})$ satisfy the weak formulation~\eqref{eqn:weak-form-1}-\eqref{eqn:weak-form-2}, then the finite-element solution $\mathfrak{q}_h \in H_{h,0}(\ddiv,\mathfrak{B})$ and $\mathfrak{p}_h\in \mathfrak{L}_h^2$ of the mixed finite-element approximation~\eqref{eqn:mixed-fem-1}-\eqref{eqn:mixed-fem-2} satisfy that
\begin{align*}
  \| \mathfrak{q} - \mathfrak{q}_h \|_{H(\ddiv,\mathfrak{B})} + \| \mathfrak{p} - \mathfrak{p}_h \|_{L^2(\mathfrak{B})}  \leq c h \left( \| \bm{q}^D \|_1 + \| \nabla \cdot \bm{q}^D \|_1^2 + \sum_{i \in \mathcal{N}_T} \| q_i^S \|_1  + \| p^D \|_1 \right),
\end{align*}
where the constant $c$ depends on $\beta$ and the quasi-uniformity of the mesh $\mathcal{M}$.
\end{corollary}

\begin{remark} \label{rem:fem-err-reg}
In Corollary~\ref{coro:fem-convergence}, we require $\nabla \cdot \bm{q}^D \in H_1(\Omega)$ because the convergence analysis is derived by following the standard Babu\v{s}ka-Brezzi theory. As it is well-known for the error analysis of the fixed-dimensional mixed finite-element method for second-order elliptic problem, this regularity requirement can be relaxed in the mixed-dimensional setting as well, i.e., we have the following error estimates
\begin{align*}
 \| \mathfrak{q} - \mathfrak{q}_h \|_{H(\ddiv,\mathfrak{B})} + \| \mathfrak{p} - \mathfrak{p}_h\|_{L^2(\mathfrak{B})} \leq c h \left( \| \bm{q}^D \|_1 + \| r^D \|_1  + \sum_{i \in \mathcal{N}_T} \| q_i^S \|_1  + \| p^D \|_1 \right),
\end{align*}
Due to space constraints, we omit the derivation here but comment that it is essentially the same as the derivation for the fixed-dimensional case as shown in~\cite{boffi2013mixed}.
\end{remark}

\section{Mass Lumping and Two-Point Flux Approximation Scheme} \label{sec:TPFA}
In practice, when the triangulation of the domain~$\Omega$ is uniform, it is possible to simply the discretization scheme and use two-point flux approximation (TPFA) to discretize the PDE system given by the conservation laws~\eqref{eqn:conservation-brain-scale}, \eqref{eqn:conservation-int-node}, \eqref{eqn:conservation-term-node-scale} and the constitutive laws~\eqref{eqn:potential-flow-brain}, \eqref{eqn:pontential-flow-network}, \eqref{eqn:potential-flow-coupling-scale}. This is particularly relevant for medical applications, where the data is frequently specified on voxels (i.e. regular Cartesian grids in 3D). 

In this section, we, therefore, discuss the TPFA scheme for our coupled Network-Darcy model through its relationship with the mixed finite-element approximation~\eqref{eqn:mixed-fem-1} and \eqref{eqn:mixed-fem-2} discussed in Section~\ref{sec:FEM}.

\subsection{TPFA Scheme} \label{subsec:TPFA}
On a given mesh $\mathcal{M}$, similar to standard diffusion problems,  the TPFA scheme can obtained by applying mass lumping to the mixed finite-element scheme~\eqref{eqn:mixed-fem-1}-\eqref{eqn:mixed-fem-2} and then eliminating the flux $\mathfrak{q}_h$.  To this end, we define the following inner product on the finite element spaces $H_h(\ddiv,\mathfrak{B})$, for $\mathfrak{q}_h$ and $\mathfrak{v}_h \in H_h(\ddiv,\mathfrak{B})$, 
\begin{eqnarray}
(\mathfrak{q}_h, \mathfrak{v}_h)_{\mathfrak{K}^{-1},h} &:= \sum_{\tau \in \mathcal{M}} \sum_{f \in \partial \tau} \omega_f \left( \bm{q}^D \cdot \bm{n}_f  \right) \left( \bm{v}^D \cdot \bm{n}_f \right)  \nonumber   \\
 & + \sum_{i\in \mathcal{N}_T} \int_{B_i} q_i^S v_i^S \, \mathrm{d}\bm{x} + \sum_{e(i,j)\in\mathcal{E}} \left(k^N_{e(i,j)}\right)^{-1} q_{i,j}^N v_{i,j}^N, \label{eqn:masslump-Hhdiv-inner}
\end{eqnarray}
where $\omega_f = \left( k^D_{\tau} \right)^{-1} \frac{d_f}{2|f|}$ with $k^D_\tau$ being the average of $k^D$ on the element $\tau \in \mathcal{M}$ and $d_f$ being the distance between the face $f \in \partial \tau$ and the cell center of $\tau$.  Now we define the mass lumping finite-element scheme as follows: Find $\mathfrak{q}_h \in H_{h,0}(\ddiv,\mathfrak{B})$ and $\mathfrak{p}_h \in L^2_h(\mathfrak{B})$, such that,
\begin{align}
&(\mathfrak{q}_h, \mathfrak{v}_h)_{\mathfrak{K}^{-1},h} - (\mathfrak{p}_h, \mathfrak{D}\cdot\mathfrak{v}_h)  = 0, \quad \forall \ \mathfrak{v}_h \in H_{h,0}(\ddiv,\mathfrak{B}), \label{eqn:masslumping-mixed-fem-1}  \\
&-(\mathfrak{D} \cdot \mathfrak{q}_h, \mathfrak{w}_h)  = -(\mathfrak{r}, \mathfrak{w}_h), \quad \forall \ \mathfrak{w}_h \in \mathfrak{L}_h^2. \label{eqn:masslumping-mixed-fem-2}
\end{align}

Based on the inner product~\eqref{eqn:masslump-Hhdiv-inner}, we define a discrete gradient $\mathfrak{D}_h : L^2_h(\mathfrak{B}) \mapsto H_h(\ddiv,\mathfrak{B})$ via integration by part (Lemma~\ref{lem:IBP}), for any $\mathfrak{v}_h \in H_h(\ddiv,\mathfrak{B})$ and $\mathfrak{p}_h \in L^2_h(\mathfrak{B})$, such that,
\begin{equation*}
(\mathfrak{D}_h \mathfrak{p}_h, \mathfrak{v}_h)_{\mathfrak{K}^{-1},h} := - (\mathfrak{p}_h, \mathfrak{D}\cdot \mathfrak{v}_h) + \int_{\partial\Omega} p^D_h\bm{v}^D_h\cdot\bm{n} \, \mathrm{d}\bm{x}  +\sum_{i \in \mathcal{N}_D} (p_h^N)_i (v_h^N)_{\mathcal{N}_i,i}.
\end{equation*}
Note that, due to the boundary conditions, $\bm{v}_h^D \cdot \bm{n} = 0$ on $\partial \Omega$ and $(p_h^N)_i = 0$, $i \in \mathcal{N}_D$, we simply have $(\mathfrak{D}_h \mathfrak{p}_h, \mathfrak{v}_h)_{\mathfrak{K}^{-1},h} = - (\mathfrak{p}_h, \mathfrak{D}\cdot \mathfrak{v}_h)$. Then the mass lumping mixed-formulation~\eqref{eqn:masslumping-mixed-fem-1} and~\eqref{eqn:masslumping-mixed-fem-2} can be written as, find $\mathfrak{q}_h \in H_{h,0}(\ddiv,\mathfrak{B})$ and $\mathfrak{p}_h \in L^2_h(\mathfrak{B})$, such that,
\begin{align*}
&(\mathfrak{q}_h, \mathfrak{v}_h)_{\mathfrak{K}^{-1},h} + (\mathfrak{D}_h\mathfrak{p}_h, \mathfrak{v}_h)_{\mathfrak{K}^{-1},h}  = 0, \quad \forall \ \mathfrak{v}_h \in H_{h,0}(\ddiv,\mathfrak{B}),  \\
&(\mathfrak{D}_h\mathfrak{w}_h, \mathfrak{q}_h)_{\mathfrak{K}^{-1},h}  = -(\mathfrak{r}, \mathfrak{w}_h), \quad \forall \ \mathfrak{w}_h \in \mathfrak{L}_h^2. 
\end{align*}
The above formulation allows us to eliminate $\mathfrak{q}_h$ and obtain the TPFA scheme as follows, find $\mathfrak{p}_h \in L^2_h(\mathfrak{B})$, such that
\begin{equation} \label{eqn:TPFA}
(\mathfrak{D}_h \mathfrak{p}_h, \mathfrak{D}_h \mathfrak{w}_h)_{\mathfrak{K}^{-1},h} = (\mathfrak{r}_h, \mathfrak{w}_h), \quad \forall \ \mathfrak{w}_h \in \mathfrak{L}_h^2. \end{equation}

Next we will explain the TPFA scheme~\eqref{eqn:TPFA} using matrix notation.  The matrix form of the mass lumping finite-element scheme~\eqref{eqn:masslumping-mixed-fem-1}-\eqref{eqn:masslumping-mixed-fem-2} can be written as
\begin{equation*}
\begin{pmatrix}
\mathsf{D}_D & \mathsf{0}   & \mathsf{0}   & \mathsf{G}_{DD} & \mathsf{0} \\
\mathsf{0}   & \mathsf{D}_S & \mathsf{0}   & \mathsf{G}_{SD} & \mathsf{G}_{SN} \\
\mathsf{0}   & \mathsf{0}   & \mathsf{D}_N & \sf{0}      & \mathsf{G}_{NN} \\
\mathsf{G}_{DD}^T & \mathsf{G}_{SD}^T & \mathsf{0} & \mathsf{0} & \mathsf{0} \\
\mathsf{0} & \mathsf{G}_{SN}^T & \mathsf{G}_{NN}^T & \mathsf{0} & \mathsf{0} 
\end{pmatrix}
\begin{pmatrix}
\mathsf{q}^D_h \\
\mathsf{q}^S_h \\
\mathsf{q}^N_h \\
\mathsf{p}^D_h \\
\mathsf{p}^N_h 
\end{pmatrix}
=
\begin{pmatrix}
\mathsf{0} \\
\mathsf{0} \\
\mathsf{0} \\
-\mathsf{r}^D \\
-\mathsf{r}^N
\end{pmatrix},
\end{equation*}
where 
$$
\sum_{\tau \in \mathcal{M}} \sum_{f \in \partial \tau} \omega_f \left( \bm{q}^D \cdot \bm{n}_f  \right) \left( \bm{v}^D \cdot \bm{n}_f \right) \mapsto \mathsf{D}_D, \ \sum_{i \in \mathcal{N}_T} \int_{B_i} q_i^S v_i^S \, \mathrm{d} \bm{x} \mapsto \mathsf{D}_S, \ \sum_{e(i,j)\in\mathcal{E}} \left( k^N_{e(i,j)} \right)^{-1} q_{i,j}^N v_{i,j}^N \mapsto \mathsf{D}_N,
$$
$$ -\int_{\Omega} p^D \, \nabla \cdot \bm{v}^D \, \mathrm{d} \bm{x} \mapsto \mathsf{G}_{DD}, \quad \sum_{i \in \mathcal{N}_T} \int_{B_i} k_i^S v_{i}^S p^D \, \mathrm{d} \bm{x} \mapsto \mathsf{G}_{SD},
$$ 
$$\sum_{i \in \mathcal{N}_T} (\int_{B_i} k_i^S v_i^S\, \mathrm{d} \bm{x})p_i^N \mapsto \mathsf{G}_{SN}, \ \text{and} \ \sum_{i\in \mathcal{N}_I \cup \mathcal{N}_N} (\sum_{j \in \mathcal{N}_i} v_{j,i}^N) p_i^h + \sum_{i \in \mathcal{N}_T} v^N_{\mathcal{N}_i,i} p^N_{i} \mapsto \mathsf{G}_{NN}.
$$ 
Since $\mathsf{D}_D$, $\mathsf{D}_s$, and $\mathsf{D_N}$ are diagonal matrices, we can eliminate them by block Gaussian elimination and end up with a linear system only involves solving for $\mathsf{p}^D_h$ and $\mathsf{p}^N_h$ as follows,
\begin{equation*}
\begin{pmatrix}
\mathsf{G}_{DD}^T & \mathsf{G}_{SD}^T & \mathsf{0}  \\
\mathsf{0} & \mathsf{G}_{SN}^T & \mathsf{G}_{NN}^T  
\end{pmatrix}
\begin{pmatrix}
\mathsf{D}_D & \mathsf{0}   & \mathsf{0}    \\
\mathsf{0}   & \mathsf{D}_S & \mathsf{0}    \\
\mathsf{0}   & \mathsf{0}   & \mathsf{D}_N
\end{pmatrix}^{-1}
\begin{pmatrix}
 \mathsf{G}_{DD} & \mathsf{0} \\
 \mathsf{G}_{SD} & \mathsf{G}_{SN} \\
 \sf{0}      & \mathsf{G}_{NN} 
\end{pmatrix}
\begin{pmatrix}
\mathsf{p}^D_h \\
\mathsf{p}^N_h 
\end{pmatrix}
= 
\begin{pmatrix}
\mathsf{r}^D \\
\mathsf{r}^N
\end{pmatrix},
\end{equation*} 
which is exactly the matrix form of the TPFA scheme~\eqref{eqn:TPFA}.

\subsection{Well-posedness}
Next we consider the well-posedness of the TPFA scheme~\eqref{eqn:TPFA}.  As we showed in the previous section, the TPFA scheme~\eqref{eqn:TPFA} is obtained from the mass lumpping mixed-formulation~\eqref{eqn:masslumping-mixed-fem-1}-\eqref{eqn:masslumping-mixed-fem-2} by block Gaussian elimination.  Therefore, we first show the well-posedness of the mass lumpping mixed-formulation~\eqref{eqn:masslumping-mixed-fem-1}-\eqref{eqn:masslumping-mixed-fem-2} and then the well-posedness of the TPFA scheme~\eqref{eqn:TPFA} follows directly.  

Since the only difference between the mixed-formulation~\eqref{eqn:mixed-fem-1}-\eqref{eqn:mixed-fem-2} and the
mass lumpping mixed-formulation~\eqref{eqn:masslumping-mixed-fem-1}-\eqref{eqn:masslumping-mixed-fem-2} is the inner product used for $H_h(\ddiv,\mathfrak{B})$, we first introduce the norm induced by the inner product~\eqref{eqn:masslump-Hhdiv-inner} as follows,
\begin{equation*}
\| \mathfrak{q}_h \|^2_{\mathfrak{K}^{-1},h} := (\mathfrak{q}_h ,\mathfrak{q}_h )_{\mathfrak{K}^{-1},h}, \quad \forall \, \mathfrak{q}_h \in H_h(\ddiv,\mathfrak{B}),
\end{equation*}
and show it is spectrally equivalent to the norm~\eqref{eqn:weighted-norm} in the following lemma.

\begin{lemma} \label{lem:norm-spectral-equivalent}
For any $\mathfrak{q}_h \in H_h(\ddiv,\mathfrak{B})$, we have
\begin{equation}\label{ine:norm-spectral-equivalent}
c_1 \| \mathfrak{q}_h \|^2_{\mathfrak{K}^{-1},h} \leq  \| \mathfrak{q}_h \|^2_{\mathfrak{K}^{-1}} \leq c_2 \| \mathfrak{q}_h \|^2_{\mathfrak{K}^{-1},h},
\end{equation}
where $c_1>0$ and $c_2>0$ are constants only depending on the shape regularity of the mesh $\mathcal{M}$.
\end{lemma}
\begin{proof}
Based on the standard result, e.g.,~\cite{Hiptmair1997}, we have
\begin{equation*}
\bar{c}_1 \sum_{\tau \in \mathcal{M}} \sum_{f \in \partial \tau} \omega_f \left( \bm{q}_h^D \cdot \bm{n}_f  \right)^2 \leq \int_{\Omega} \left(k^D \right)^{-1} |\bm{q}^D_h|^2 \, \mathrm{d} \bm{x} \leq \bar{c}_2 \sum_{\tau \in \mathcal{M}} \sum_{f \in \partial \tau} \omega_f \left( \bm{q}_h^D \cdot \bm{n}_f  \right)^2
\end{equation*}
where  the positive constants $\bar{c}_1$ and $\bar{c}_2$ depend only the shape regularity of the mesh $\mathcal{M}$. Then the spectral equivalence~\eqref{ine:norm-spectral-equivalent} follows directly from the definitions of the norms. 
\end{proof}

Define 
\begin{equation}\label{eqn:norm-Hhdiv}
\| \mathfrak{q}_h \|_{H_h(\ddiv,\mathfrak{B})}^2 := \| \mathfrak{q}_h \|_{\mathfrak{K}^{-1},h}^2 + \| \mathfrak{D}\cdot \mathfrak{q}_h \|_{L^2(\mathfrak{B})}^2.
\end{equation}
We have the following lemmas concerning the continuity, ellipticity, and inf-sup condition for the mass lumping mixed-formulation~\eqref{eqn:masslumping-mixed-fem-1}-\eqref{eqn:masslumping-mixed-fem-2}.

\begin{lemma}[Continuity of~\eqref{eqn:masslumping-mixed-fem-1}-\eqref{eqn:masslumping-mixed-fem-2}] \label{lem:masslump-fem-continuity}
	For any $\mathfrak{q}_h$, $\mathfrak{v}_h \in H_h(\ddiv,\mathfrak{B})$ and $\mathfrak{w}_h \in \mathfrak{L}_h^2$, we have
	\begin{align*}
	(\mathfrak{q}_h, \mathfrak{v}_h)_{\mathfrak{K}^{-1},h} & \leq \| \mathfrak{q}_h \|_{H_h(\ddiv,\mathfrak{B})} \| \mathfrak{v}_h \|_{H_h(\ddiv,\mathfrak{B})}, \\
	(\mathfrak{D} \cdot \mathfrak{q}_h, \mathfrak{w}_h) & \leq   \| \mathfrak{q}_h \|_{H_h(\ddiv,\mathfrak{B})} \| \mathfrak{w}_h \|_{L^2(\mathfrak{B})}.
	\end{align*}
\end{lemma}

For the ellipticity, again using the fact that, for $\mathfrak{q}_h \in H_h(\ddiv,\mathfrak{B})$, $\mathfrak{D}\cdot\mathfrak{q}_h \in \mathcal{L}^2_h$, we have
\begin{lemma}[Ellipticity of~\eqref{eqn:masslumping-mixed-fem-1}-\eqref{eqn:masslumping-mixed-fem-2}] \label{lem:masslump-fem-ellipticity}
	If $\mathfrak{q}_h \in H_h(\ddiv,\mathfrak{B})$ satisfies
	\begin{equation*} 
	(\mathfrak{D}\cdot \mathfrak{q}_h, \mathfrak{w}_h) = 0, \quad \forall \, \mathfrak{w}_h \in \mathfrak{L}_h^2,
	\end{equation*} 
	then
	\begin{equation*} 
	(\mathfrak{q}_h, \mathfrak{q}_h)_{\mathfrak{K}^{-1},h} = \| \mathfrak{q}_h \|_{H_h(\ddiv,\mathfrak{B})}^2
	\end{equation*}
\end{lemma}

Moreover, the inf-sup condition  can be derived from the inf-sup condition (Lemma~\ref{lem:fem-inf-sup}) and the spectral equivalence lemma (Lemma~\ref{lem:norm-spectral-equivalent})
\begin{lemma}[Inf-sup condition of~\eqref{eqn:masslumping-mixed-fem-1}-\eqref{eqn:masslumping-mixed-fem-2}] \label{lem:masslump-fem-inf-sup}
	There exists a constant $\beta>0$ such that, for any given function $\mathfrak{r}_h\in L^2(\mathfrak{B})$,
	\begin{equation} \label{ine:masslump-fem-inf-sup}
	\sup_{\mathfrak{q}_h\in H_{h,0}(\ddiv,\mathfrak{B})} \frac{(\mathfrak{r}_h, \mathfrak{D} \cdot \mathfrak{q}_h)}{ \| \mathfrak{q}_h \|_{H_h(\ddiv,\mathfrak{B})} } \geq \beta \| \mathfrak{r}_h \|_{L^2(\mathfrak{B})}.
	\end{equation}
	Here, the inf-sup constant $\beta$ depends on $|\mathcal{M}_i| = \operatorname{measure}(\mathcal{M}_i) = \mathcal{O}(h^n)$, the maximal number of overlaps between $B_i$, structure of the trees $\mathcal{T} \in \mathcal{F}$, the domain $\Omega$, the constants $c_{k^S}$ and $C_{k^S}$, and the shape regularity of the mesh $\mathcal{M}$. 
\end{lemma}
\begin{proof}
The inf-sup condition~\eqref{ine:masslump-fem-inf-sup} can be derived from the inf-sup condition~\eqref{ine:fem-inf-sup} and the spectral equivalence result~\eqref{ine:norm-spectral-equivalent}. 
\end{proof}

Now the well-posedness of the mass lumping mixed formulation~\eqref{eqn:masslumping-mixed-fem-1} and~\eqref{eqn:masslumping-mixed-fem-2} follows from Lemmas~\ref{lem:masslump-fem-continuity}, \ref{lem:masslump-fem-ellipticity}, and~\ref{lem:masslump-fem-inf-sup}.
\begin{theorem}[Well-posedness of~\eqref{eqn:masslumping-mixed-fem-1}-\eqref{eqn:masslumping-mixed-fem-2}] \label{thm:masslump-fem-wellposedness}
	The mass lumping mixed formulation~\eqref{eqn:masslumping-mixed-fem-1}-\eqref{eqn:masslumping-mixed-fem-2} is well-posed with respect to the norms~\eqref{eqn:norm-Hhdiv} and~\eqref{eqn:norm-L2}.
\end{theorem}

Finally, the well-posedness of the TPFA scheme~\eqref{eqn:TPFA} follows directly from Theorem~\eqref{thm:masslump-fem-wellposedness} and the equivalence between the TPFA scheme~\eqref{eqn:TPFA} and the mass lumpping mixed-formulation~\eqref{eqn:masslumping-mixed-fem-1}-\eqref{eqn:masslumping-mixed-fem-2}.  The result is summarized in the following theorem.

\begin{theorem}[Well-posedness of~\eqref{eqn:TPFA}] \label{thm:TPFA-wellposedness}
The TPFA scheme~\eqref{eqn:TPFA} is well-posed. 
\end{theorem}

\subsection{Convergence}
Regarding the convergence result of the TPFA scheme, since we use mass-lumping technique to derive it, existing theoretical tools developed in~\cite{BarangerMaitreOudin1996,BrezziFortinMarini2006} can be adopted here.  For the sake of the simplicity, in this subsection, we assume that $k^D$ is constant on each element $\tau \in \mathcal{M}$ and the mesh $\mathcal{M}$ is uniform  (e.g., rectangle/equilateral triangle in 2D, rectangular cuboid/regular tetrahedra in 3D).  Under those conditions, as shown in~\cite{BarangerMaitreOudin1996}, for $\tau \in \mathcal{M}$, $\sum_{f \in \partial \tau} \omega_f \left( \bm{q}^D \cdot \bm{n}_f  \right) \left( \bm{v}^D \cdot \bm{n}_f \right)$ used in the definition~\eqref{eqn:masslump-Hhdiv-inner} provides a numerical integration formula of $\int_{\tau} (k^D)^{-1} \bm{q}^D \bm{v}^D \, \mathrm{d} \bm{x}$ and such a numerical integration is exact for constant functions on each element $\tau$.  Moreover, the following perturbation result holds for $\bm{q}^D, \bm{v}^D \in H_{h,0}(\ddiv, \mathcal{M})$,
\begin{equation}\label{ine:mass-lump-perturb}
| \int_{\tau} (k^D)^{-1} \bm{q}^D \bm{v}^D \, \mathrm{d} \bm{x} - \sum_{f \in \partial \tau} \omega_f \left( \bm{q}^D \cdot \bm{n}_f  \right) \left( \bm{v}^D \cdot \bm{n}_f \right) | \leq c h_{\tau}^2 \| \bm{q}^D \|_{H(\ddiv, \tau)} \| \bm{v}^D \|_{H(\ddiv, \tau)}.
\end{equation}
Based on the above result, we can easily verify that, for $\mathfrak{q}_h$, $\mathfrak{v}_h \in H_h(\ddiv,\mathfrak{B})$, 
\begin{equation*}
|  (\mathfrak{q}_h, \mathfrak{v}_h)_{\mathfrak{K}^{-1}}  - (\mathfrak{q}_h, \mathfrak{q}_h)_{\mathfrak{K}^{-1},h}  |  \leq c h^2 \| \mathfrak{q}_h \|_{H(\ddiv,\mathfrak{B})} \| \mathfrak{v}_h \|_{H(\ddiv,\mathfrak{B})}.
\end{equation*}
Now, we can use the theory developed in~\cite{RobertsThomas1991} and conclude the convergence result of the TPFA scheme in the following theorem. 
\begin{theorem}
Suppose that $\mathfrak{q} \in H_0(\ddiv,\mathfrak{B})$ and $\mathfrak{p}\in L^2(\mathfrak{B})$ satisfy the weak formulation~\eqref{eqn:weak-form-1}-\eqref{eqn:weak-form-2}, then the finite-element solution $\mathfrak{q}_h \in H_{h,0}(\ddiv,\mathfrak{B})$ and $\mathfrak{p}_h\in \mathfrak{L}_h^2$ of the mass lumping mixed finite-element approximation~\eqref{eqn:masslumping-mixed-fem-1}-\eqref{eqn:masslumping-mixed-fem-2} satisfy that
\begin{align}
\| \mathfrak{q} - \mathfrak{q}_h\|_{H(\ddiv,\mathfrak{B})} + \| \mathfrak{p} - \mathfrak{p}_h \|_{L^2(\mathfrak{B})} \leq c h \left(  \| \bm{q}^D \|_1 + \| \nabla \cdot \bm{q}^D \|_1 + \sum_{i \in \mathcal{N}_T} \| q_i^S \|_1  + \| p^D \|_1 \right) \label{ine:masslumping-error-estimate}
\end{align}
where the constant $c$ depends only on $\beta$, $k^D$, the maximal number of the overlap between $\mathcal{M}_i$, $\max_{i}\{|\mathcal{M}_i|\}$, and quasi-uniformity of the mesh $\mathcal{M}$.
\end{theorem}

Consequentially, this also implies the convergence result of the TPFA scheme because of the equivalence between the TPFA scheme~\eqref{eqn:TPFA} and the mass lumpping mixed-formulation~\eqref{eqn:masslumping-mixed-fem-1} and~\eqref{eqn:masslumping-mixed-fem-2}.

\begin{remark}
As pointed out in Remark~\ref{rem:fem-err-reg}, the regularity requirement $\nabla \cdot \bm{q}^D \in H_1(\Omega)$ can be relaxed here as well and similar convergence analysis still holds.  
\end{remark}

\begin{remark}
As shown in~\cite{BarangerMaitreOudin1996,BrezziFortinMarini2006}, similar results hold for some more general meshes.  For example, the perturbation result~\eqref{ine:mass-lump-perturb} hold for general triangles in 2D with order $h$ instead of order $h^2$. However, this still leads to the error estimate~\eqref{ine:masslumping-error-estimate} based on the same procedure.  For general triangulation in 3D, convergence analysis for standard mixed-formulation Poisson problem with mass lumping was derived based on a different approach in~\cite{BrezziFortinMarini2006}.  We can also adopt a similar approach to derive the convergence result for our mass lumping mixed finite-element scheme as well to obtain the error estimate~\eqref{ine:masslumping-error-estimate} for general triangulation as well. 
\end{remark}

\section{Numerical Results} \label{sec:numerics}
In this section, we include three numerical results to validate and explore the discretization and solver presented above. In particular, the first case contains the simplest possible geometry in 2D, on which we compare the discretization to a series solution (Bessel functions). In the second case, we have a more complex geometry embedded in 4D, which can be seen as a prototype of the geometries relevant for applications. In both the first and second cases, we perform convergence studies both for the discretization and multigrid solver. Finally, in the third case, we apply the methodology to a real dataset, based on the human brain. 

The error is measured in the norms proposed in the analysis, in particular we measure the $L^2$ norm of pressure and the $k^{-1/2}$-weighted norm of flux. As is common for finite volume and mixed finite-element methods, we use cell-centered quadrature when evaluating the $L^2$ norm in the domain, which allows us to exhibit the usual super-convergence behavior for these methods on smooth problems.

Due to the prevalence of image data for the applications of interest, all the numerical experiments are conducted on uniform Cartesian grids and the TPFA scheme is used.  To solve the resulting linear system, we use algebraic multigrid (AMG) preconditioned flexible GMRes (FGMRes) method, as detailed in the Supplementary Materials~\ref{sec:solvers}. Here, an unsmoothed aggregation AMG method is used as the preconditioner.  More precisely, one step of V-cycle AMG method is applied with one step of Gauss-Seidel method for both pre- and post-smoothing.  The FGMRes method is terminated when the $\ell^2$-norm of the initial residual is reduced by a factor of $10^{-6}$. The solver performance for all three cases below is also reported in the Supporting Information. 
The implementations are in Matlab, and code is available from the authors on request. All runs are conducted on a Linux workstation using 40 Intel Xeon CPU processors (E5-2698 v4) at 2.20GHz clock speed, with 256 Gb RAM.

\subsection{Case 1: Comparison to Convergent Series Solution}\label{sec:Case1}

Our first case is constructed such that a series solution (in terms of well-known Bessel functions) is available. The full derivation of the series solution is available in the Appendix, an illustration of the geometry, and the series solution is provided in Figure \ref{fig:case1_illustration}. Throughout this subsection, we consider the series solution as the exact solution of the equations, since arbitrary precision can be obtained using well-established implementations of table values \cite{abramowitz1948handbook}. 

\begin{figure}
    \centering
    \includegraphics[width=9cm]{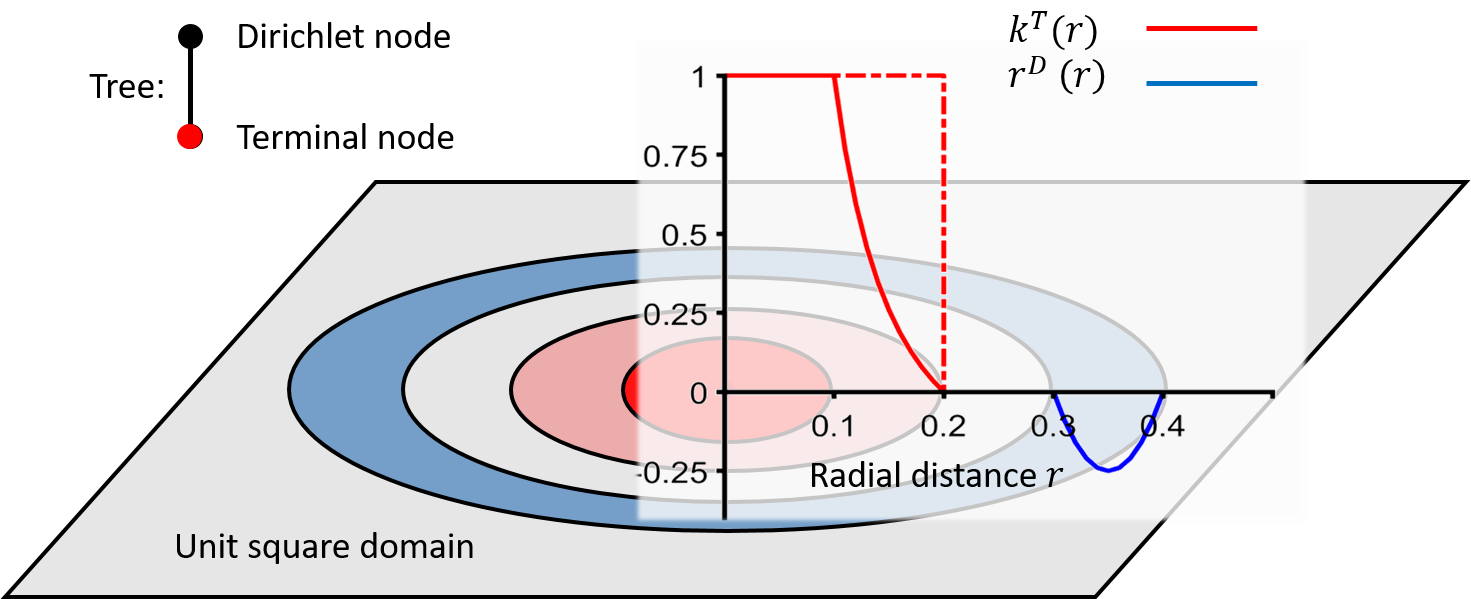}
    \includegraphics[width=6cm]{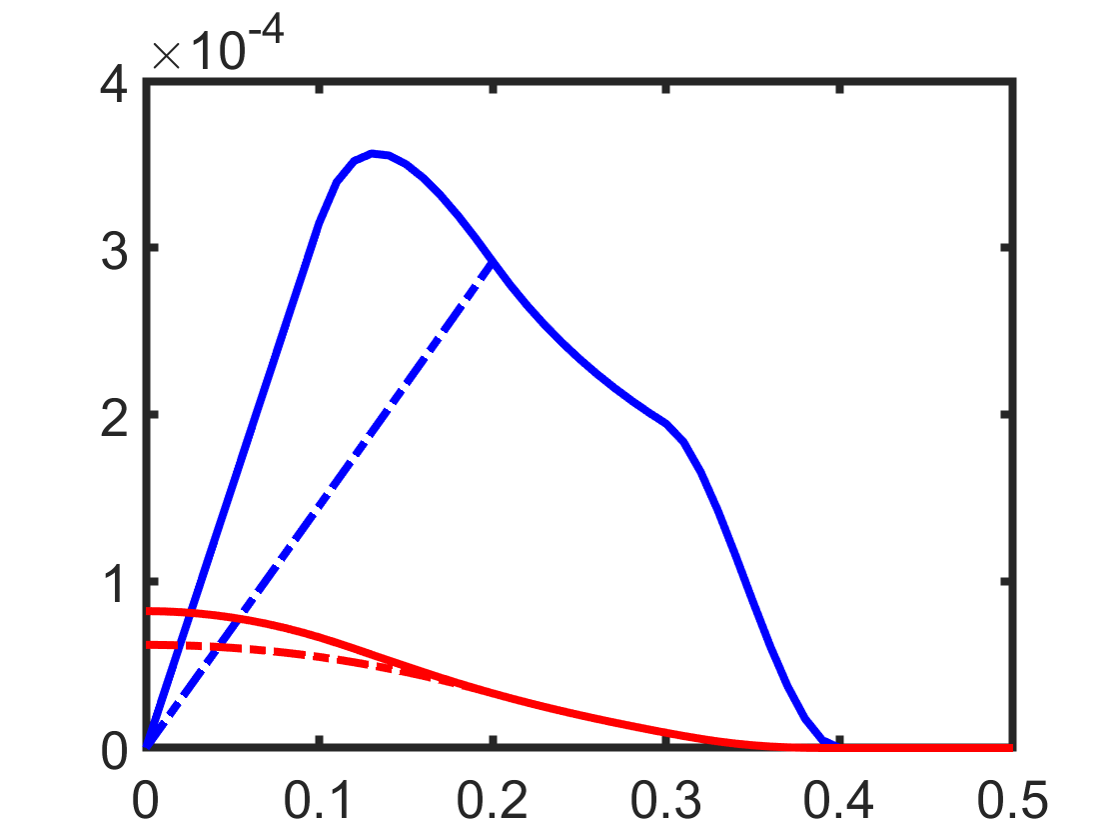}
    \caption{Left: Illustration of domain for Case 1, with transfer function $k^T$ (red) and source term $r^D$ (blue). The source term, which is actually a sink in this setup, has been scaled by $10^2$ for visualization purposes. Right: Illustration of pressure (black) and radial flux (grey) in the domain as function of distance from the origin. Note that for the pressure, we have plotted $p^D(r) - p^D(0.5)$ in order to facilitate visual comparison. In both figures, cases 1A is represented by solid lines and 1B by dashed-dotted lines. }
    \label{fig:case1_illustration}
\end{figure}

The main features of the solution is a simple two-node tree, where node 0 is a Dirichlet boundary node, and node 1 is a terminal node. Correspondingly, there is a single edge in the network, which contains the network flux. The solution is constructed with a transfer function $k^T$ that has compact support on a disc of radius $r_1$ from the origin. We consider two variants of the case, case 1A has a smoothly degenerating transfer function such that (in terms of radial coordinates) $k^T(r) \rightarrow 0$ as $r\rightarrow r_1$, while case 1B has a constant $k^T$ within the disc (and zero outside), thus $k^T \sim H(r_1-r)$, where $H$ denotes the Heaviside function. To drive the system, a quadratic source term is provided in the region $r_2 < r \leq r_3$. 

We conduct numerical experiments with unit values, such that the domain $\Omega$ is the unit square centered at the origin, the domain and network permeabilities are unit valued, and the scaling of source term $r^D=1$. The transfer function $k^T$ has a unit maximum value at the origin, for both case A and B, thus in the notation of the appendix $k^T_0=1$. As stated, we consider two versions of the case. For the case 1A, we consider a degenerating transfer function $k^T$, with $r_0 = 0.1$, $r_1 = 0.2$, $r_2 = 0.3$, $r_3 = 0.4$. For case 1B, we let the transfer function abruptly go to zero by keeping all radii as in case 1A, except for $r_0 = 0.2$. 

An important aspect of the implementation is the accuracy with which the right-hand-side and the inner products involving $k^S$ are evaluated. In the results reported here, we have used a fourth-order accurate numerical quadrature. 

The convergence results of cases 1A and 1B are presented in Table \ref{tab:Case1A} and \ref{tab:Case1B}. We show the convergence history separated into components similar to the analysis, i.e. Domain, Scaled terminal flux, and Network. 

First note that for this example, since the network contains a single throat and the domain has Neumann boundary conditions, global conservation of mass implies that $q_h^N$ will be exact up to the quadrature error in the evaluation of $r^D$, and similarly for $p_h^N$. Thus the fourth-order convergence of these variables is expected. 

As for the remaining variables, we observe in both Case 1A and Case 1B optimal second-order convergence of $p_h^D$ and first-order convergence of $\bm{q}_h^D$.  In this example, the scaled terminal flux $q_h^S$ is essentially just the weighted difference between $p_h^D$ and $p_h^N$, and thus it inherits the (slower) convergence rate of the two, i.e. second-order. By comparing the two cases, we see that there is no influence of the degeneracy of $k^S$.

\begin{table}[]
\scriptsize
\begin{center}
\begin{tabular}{ l  l  l  l  l  l  l  l }
Variable & $1/h$ & Error $D$ & Rate $D$ & Error $S$ & Rate $S$ & Error $N$ & Rate $N$\\
\hline
 &     $16$ & 1.81e-07 &     &     &     & 4.91e-09 &    \\
 &     $32$ & 4.12e-08 & 2.13 &     &     & 1.59e-10 & 4.95\\
$p$ &     $64$ & 1.03e-08 & 1.99 &     &     & 1.23e-11 & 3.70\\
 &     $128$ & 2.63e-09 & 1.98 &     &     & 3.69e-13 & 5.06\\
 &     $256$ & 6.55e-10 & 2.00 &     &     & 4.88e-15 & 6.24\\
 &     $512$ & 1.64e-10 & 2.00 &     &     & 2.59e-16 & 4.24\\
Average &  &  & 2.02 &  &     &  & 4.84\\
\hline
 &     $16$ & 1.68e-05 &     & 2.38e-07 &     & 4.91e-09 &    \\
 &     $32$ & 8.29e-06 & 1.02 & 4.98e-08 & 2.26 & 1.59e-10 & 4.95\\
$q$ &     $64$ & 4.11e-06 & 1.01 & 1.25e-08 & 1.99 & 1.23e-11 & 3.70\\
 &     $128$ & 2.06e-06 & 0.99 & 3.05e-09 & 2.04 & 3.69e-13 & 5.06\\
 &     $256$ & 1.03e-06 & 0.99 & 7.64e-10 & 2.00 & 4.88e-15 & 6.24\\
 &     $512$ & 5.19e-07 & 0.99 & 1.91e-10 & 2.00 & 2.59e-16 & 4.24\\
Average &  &  & 1.00 &  & 2.06 &  & 4.84\\
\hline
\end{tabular}
\end{center}
\caption{Convergence of case 1A. Upper part of the table gives convergence information for the pressure variables $p^D$ and $p^N$, while the lower part of the table gives the convergence information for the flux variables ${\bm q}^D$, $q^S$ and $q^N$.}\label{tab:Case1A}
\label{}
\end{table}

\begin{table}[H]
\scriptsize
\begin{center}
\begin{tabular}{ l  l  l  l  l  l  l  l }
Variable & $1/h$ & Error $D$ & Rate $D$ & Error $S$ & Rate $S$ & Error $N$ & Rate $N$\\
\hline
 &     $16$ & 2.02e-07 &     &     &     & 4.91e-09 &    \\
 &     $32$ & 3.37e-08 & 2.59 &     &     & 1.59e-10 & 4.95\\
$p$ &  $64$ & 8.06e-09 & 2.06 &     &     & 1.23e-11 & 3.70\\
 &     $128$ & 2.03e-09 & 1.99 &     &     & 3.55e-13 & 5.11\\
 &     $256$ & 5.94e-10 & 1.77 &     &     & 2.69e-15 & 7.05\\
 &     $512$ & 1.37e-10 & 2.11 &     &     & 4.88e-16 & 2.46\\
Average &  &  & 2.11 &  &     &  & 4.65\\
\hline
 &     $16$ & 1.65e-05 &     & 1.35e-06 &     & 4.91e-09 &    \\
 &     $32$ & 8.54e-06 & 0.95 & 2.00e-07 & 2.76 & 1.59e-10 & 4.95\\
$q$ &     $64$ & 4.21e-06 & 1.02 & 3.02e-08 & 2.73 & 1.23e-11 & 3.70\\
 &     $128$ & 2.11e-06 & 1.00 & 7.70e-09 & 1.97 & 3.55e-13 & 5.11\\
 &     $256$ & 1.05e-06 & 1.00 & 6.54e-10 & 3.56 & 2.69e-15 & 7.05\\
 &     $512$ & 5.26e-07 & 1.00 & 1.84e-10 & 1.83 & 4.88e-16 & 2.46\\
Average &  &  & 1.00 &  & 2.57 &  & 4.65\\
\hline
\end{tabular}
\end{center}
\caption{Convergence of case 1B. For complete legend, see figure \ref{tab:Case1A}}\label{tab:Case1B}
\end{table}

\subsection{Case 2: A Prototypical 4 Dimensional Case}
Our second example is chosen to illustrate a typical case encountered in the modeling of tissue. The physical domain is 3-dimensional, however, due to the biomedical properties involved, the physical domain represents two or more continua (biomedically speaking, this corresponds to arterial and venal compartments, etc.). The continua are ordered, and communication between the compartments is only allowed between neighbors in the ordering. As such, the continua represent a discretization of an elliptic equation in a fourth dimension. The mathematical structure of the resulting system is thus one of a 4D elliptic equation, coupled to networks, and is naturally covered by the methods proposed analyzed in this paper. 

\begin{figure}
    \centering
    \includegraphics[width=13cm]{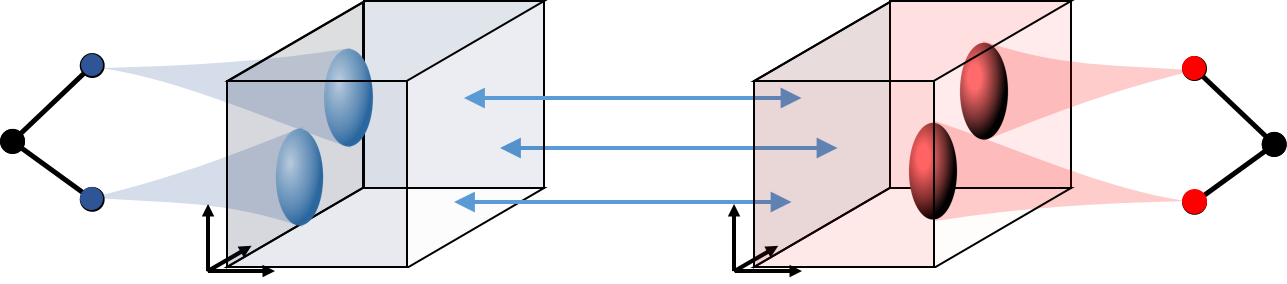}
    \caption{Illustration of domain for Case 2. The arterial network and arterial 3D domain is shaded by red colors, while the venous network and network 3D domain are shaded by blue colors. The two 3D domains together for a 2-point discretization of a 4D domain, where the flow in the fourth dimension is indicated by arrows between the two 3D domains.}
    \label{fig:case2_illustration}
\end{figure}

To explore this concept, and validate the performance of our methods, we consider the following concrete problem, as illustrated schematically in Figure \ref{fig:case2_illustration}. Let the model domain be the unit 4-cube. We consider Neumann boundary conditions on all faces of the domain. Furthermore, we consider two trees, which are named as "arterial tree" and "venous tree", respectively, to conform with applications and the next subsection. Each consists of four nodes connected in the shape of a "Y", wherein each tree, one node is a Dirichlet boundary node ($p^N_D=1$ and $p^N_D=0$ in arterial and venous Dirichlet nodes, respectively), while two nodes are terminal nodes. The arterial terminal nodes $i$ are associated with transfer functions $k_i^T(x) = k^T(|x-y_i|_3)H(1/2 - x_4)$, where $|x-x_i|_3^2 =\sum_{j=1...3} (x_j-y_{i,j})^2$ is the distance in the first three coordinates from the 3-points $y_i$, $x_4$ is the coordinate in the fourth dimension, and $k^T$ are the transfer functions from Section \ref{sec:Case1} with $r_0 =0.1$ and $r_1 = 0.2$. Conversely, the venous terminal nodes are associated with transfer functions $k_i^T(x) = k^T(|x-y_i|_3)H(x_4-1/2)$. For the arteries, the transfer functions are centered on 3-points $y_i$ defined by $[0.43, 0.25, 0.5]$ and $[0.37, 0.75, 0.5]$, while for the veins, the transfer functions are centered on $[0.63, 0.25, 0.5]$ and $[0.57, 0.75, 0.5]$.

We discretize the domain with an anisotropic Cartesian grid in the sense that the first three dimensions are discretized by a regular isotropic Cartesian grid. The fourth dimension is discretized by only two grid cells. This resulting system is equivalent to the common two-compartment model, where the cells in the fourth dimension with $x_4< 0.5$ correspond to the arterial compartment, and the remaining cells the venous compartment. In accordance with the practice in applications, we will emphasize grid refinement over model refinement, and only consider refinement of the first three dimensions. Moreover, we will in accordance with the applications decompose the domain flux into two parts $q^D\rightarrow [q^D , q^P]$, where the flux in the fourth dimension $q^P$ is referred to as "perfusion". Model parameters are otherwise set to unity, $k^D = k^P=k^N = 1$ where $k^P$ is the permeability constant of the flux in the fourth direction.

The convergence results for this case are presented in Table \ref{tab:Case2}. All errors are reported relative to a numerical solution calculated with a resolution of $h = 256^{-1}$, and convergence rates are therefore reported for grids up to a resolution of $h = 128^{-1}.$ As expected, we observe quasi-optimal convergence rates in all variables. In contrast to case 1, we no longer have the artificial exact solutions in the network, where we observe second order convergence, as inherited from the interaction between the terminal nodes and the second-order accurate pressure in the domain. 

\begin{table}[H]
\scriptsize
\begin{center}
\begin{tabular}{ l  l  l  l  l  l  l  l  l  l }
Variable & $1/h$ & ErrorD & RateD & ErrorT & RateT & ErrorN & RateN & ErrorP & RateP\\
\hline
p &    $16$ & 3.42e-05 &     &      &     & 3.86e-06 &     &      &    \\
 &     $32$ & 8.60e-06 & 1.99 &      &     & 6.20e-07 & 2.64 &      &    \\
 &     $64$ & 2.21e-06 & 1.96 &      &     & 1.84e-07 & 1.75 &      &    \\
 &     $128$ & 6.36e-07 & 1.80 &      &     & 4.20e-08 & 2.13 &      &    \\
Average &  &  & 1.92 &  &     &  & 2.17 &  &    \\
q &     $16$ & 1.12e-03 &     & 2.45e-02 &     & 2.01e-06 &     & 2.49e-05 &    \\
 &     $32$ & 4.09e-04 & 1.45 & 1.11e-02 & 1.14 & 3.24e-07 & 2.64 & 6.25e-06 & 1.99\\
 &     $64$ & 1.55e-04 & 1.40 & 5.02e-03 & 1.15 & 9.60e-08 & 1.75 & 1.60e-06 & 1.97\\
 &     $128$ & 4.99e-05 & 1.64 & 2.03e-03 & 1.31 & 2.20e-08 & 2.12 & 4.42e-07 & 1.85\\
Average &  &  & 1.49 &  & 1.20 &  & 2.17 &  & 1.94\\
\hline
\end{tabular}
\end{center}
\caption{Convergence results for Case 2. All variables are reported as in table \ref{tab:Case1A}, with the (perfusion) flux in the fourth dimension additionally reported as $q^P$.}
\label{tab:Case2}
\end{table}

\subsection{Case 3: Full-brain simulation study}
\label{sec:Case3}
As a final test case, we consider the application to a real data set, associated with blood flow in the human brain. As a modeling concept, we use the same general structure as illustrated in Figure \ref{fig:case2_illustration}. The data-set and parameterization is described in detail in  \cite{HodnelandHansonSaevareidNaevdalLundervoldSolteszovaMunthe-KaasDeistungReichenbachNordbotten2019}, and is illustrated in Figure \ref{fig:brain-figure} of the introduction.  Here we summarize the main features: the data contains two trees, corresponding to a segmentation of the arterial and venous systems, containing 355 and 1222 nodes, respectively. For the finest simulations, we consider full resolution MRI acquisitions, which after co-registration to the finest resolution image is a Cartesian grid with $346 \times 448 \times 319$ grid cells, representing a brick-shaped field of view of $177 \times 224 \times 160 \mathrm{mm}^3$. The actual domain $\Omega$ is a 4D extrusion of the 3D subset of the field of view from a T1-weighted MR acquisition which contains segmentation of the brain acquired with the human brain segmentation software FreeSurfer \cite{dale:99}. Thus the mathematical formulation is a 4D model in the sense of the previous sub-section, and after discretizing the fourth dimension by two cells, the full model contains $17.5$ million grid cells. The domain $\Omega$ is furthermore divided into two subdomains by the FreeSurfer segmentation (anatomically: white matter $\Omega_{WM}$ and gray matter $\Omega_{GM}$), with permeability in the three physical dimensions set to an isotropic value of $k^D = 10^{-11} m^2$. The permeability $k^D$ acting in the 4th dimension (anatomically: the perfusion coefficient), is anisotropic relative to the physical dimensions, and is in the white matter set to $k^P = 10^{-6}m\cdot s \cdot \mathrm{kg}
^{-1}$, $x \in \Omega_{WM}$, and  in grey matter is set to $k^P = 1.6 \cdot 10^{-6} m\cdot s \cdot \mathrm{kg}^{-1}$, $x \in \Omega_{GM}$. The transfer permeability is set according to equation \eqref{eqn:kT-form}, with $r_1 = 30$mm, $r_0 = r_1/2$, and $k_0^T=10^{-4}$. 

The arterial and venous vessel trees are extracted down to voxel resolution from time-of-flight (TOF) and quantitative susceptibility mapping (QSM), respectively. Within both these MR acquisitions, a crude segmentation of the vessels is obtained by local adaptive thresholding, leading to a large number of disconnected structures. These binary satellites are connected with the main structure by repeatedly solving a boundary value problem around the main structure $S$ for each satellite. Hence, the solution of the Eikonal equation $|\nabla T| = f(x)^{-1}, T(x \in S) = 0$ for the arrival time $T(x)$ provides a geodesic distance map from $x$ to the main structure. The Eikonal equation was solved using the fast marching method \cite{Sethian1996}. The function $f(x)$ is user-supplied and is known as the speed of the arrival time field. Within TOF we use the image itself as speed function and for QSM the inverted image due to low contrast within vessels. The speed is set to zero outside the brain, possibly leading to curved geodesic trajectories, which is the reason why the signed distance function is not used. The arrival time itself is not of interest here, but rather the backtracing in the arrival time field from the satellite to the main structure providing a most probable path connecting these two structures with each other. The current approach favors probable paths to be aligned with dark- or bright-contrast structures that partly disappear within the images due to noise in the data. While backtracing, visited points are added to the main structure with a suitable vessel radius. Finally, the process of solving the Eikonal equation is repeated for each satellite, ultimately providing a connected structure, i.e. the arterial or venous vessel tree. For a more comprehensive description of how satellites are connected to the main structure, we refer the readers to \cite{HodnelandHansonSaevareidNaevdalLundervoldSolteszovaMunthe-KaasDeistungReichenbachNordbotten2019}.

The now connected binary trees are converted into abstract graphs using built-in Matlab routines for skeletonization, leaf (terminals and roots), and node detection. Vessel length is the geodesic distance along the edge between two connecting nodes, and the average vessel diameter is fitted by a Euclidean distance function around the centerline. The edge flow permeability $k^N$ is assigned individually for each edge based on Hagen-Poiseulle's law, using local estimates of vessel diameter and vessel length measured in the binary vessel trees. Both arterial and venous trees are modeled with Dirichlet root nodes as the main arterial inlets and main venous outlets. The only properties of the vessel trees that are needed for the simulation experiments are the edge flow permeability $k^N$ within a connectivity matrix and the terminal positions within the field of view.

\begin{figure} 
    \centering
    \includegraphics[width=16cm]{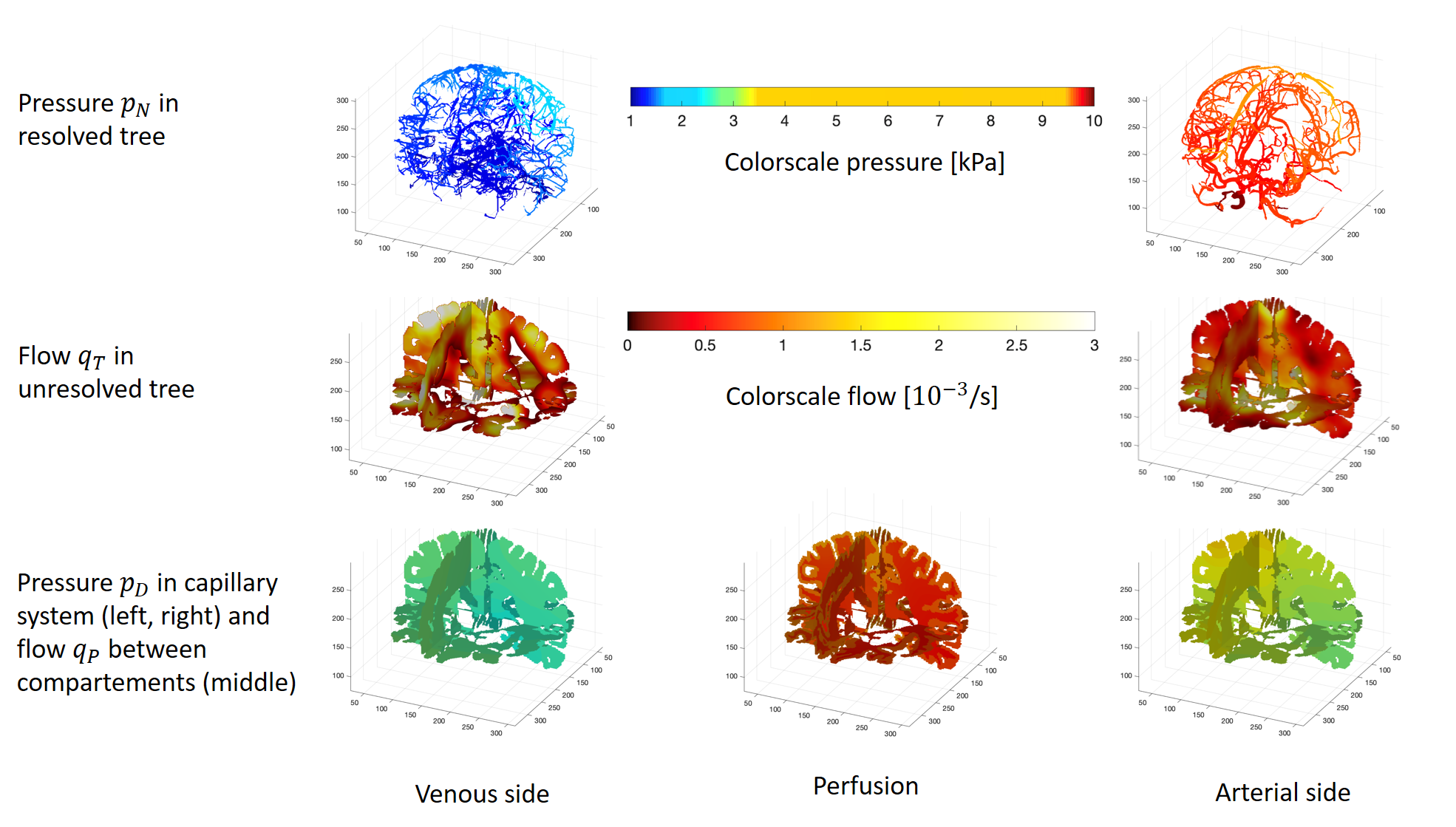}
    \caption{Simulation results for Case 3, showing the pressure solution for the trees and domain, as well as the transfer flux $q^D$ and the component of the domain flux associated with the fourth dimension, denoted $q^P$ in the text. Full-size versions of the subfigures are available in section SM3 of the Supporting Information.}
    \label{fig:case3_results}
\end{figure}

The full brain data contains many important qualitative properties, including connectivity of the trees after preprocessing of the initially disjoint trees, and connectivity of the brain geometry. These properties ensuring well-posedness, as well as the connected representation of grey and white matter, are not trivially preserved when coarsening the data. Thus instead of reporting relative results on a grid sequence for this case (which due to the above would have limited real value), we summarize the calculated solution on the image resolution in Figure \ref{fig:case3_results} (the subfigures of this figure are shown in full size in section SM3 of the supporting information). While the quantitative aspects of the calculated results depend on parameters that are at present not fully justified by clinical measurements, our calculations verify that the proposed methods allow for efficient simulations at imaging resolution, preserving the qualitative properties of the solution corresponding to biomedical expectations. 

\section{Conclusions}\label{sec:conclusion}

We have proposed a mixed-dimensional mathematical model, closely related to models used for modeling fluid flow in human vasculature. We show the well-posedness of this model on the continuous level and develop suitable numerical discretizations, of both mixed finite-element and finite volume types. These are shown to be stable and convergent. 

Our theoretical results are complemented by numerical examples, which demonstrate super-convergence of the method in terms of the pressure variable on smooth solutions, and also verifies the stability and applicability of the method to large scale real-world data sets.

\bibliographystyle{siamplain}
\bibliography{BloodFlowBrain.bib}

\end{document}


\maketitle
\section{Linear Solvers} \label{sec:solvers}
In this section, we discuss how to solve the linear systems obtained by the TPFA scheme from the main part of the paper. The TPFA scheme solves the unknown $\mathfrak{p}_h$ first, after which the unknowns $\mathfrak{q}_h$ can be recovered.  Therefore, we discuss multigrid methods for the TPFA scheme. In this section, we will solely be dealing with the discrete system, and thus omit the subscripts on $\mathfrak{q}_h$ and $\mathfrak{p}_h$.

The TPFA scheme~\eqref{eqn:TPFA}, following the discussion in Section~\ref{subsec:TPFA},  can be obtained by block Gaussion elimination of the mass lumping mixed-formulation, i.e., its compact matrix form is
\begin{equation} \label{eqn:TPFA-poisson}
\mathfrak{G}^T (\mathfrak{D}_{\bm{\mathfrak{q}}})^{-1} \mathfrak{G} \mathfrak{p} = \mathfrak{r}.
\end{equation}
As we can see, this is a discretization of a Poisson type problem for the unknown $\mathfrak{p}$ in the mixed-dimensional setting.  Since multigrid (MG) methods are designed for solving diffusion type problems efficiently~\cite{TrottenbergOosterleeSchuller2001,XuZikatanov2017}, in our work, we use an algebraic variant of the MG methods to solve the resulting linear system. More precisely, we use the aggregation-based AMG method, which is a suitable choice for solving both Poisson-type problems, see~\cite{Blaheta1986,KimXuZikatanov2003}, and graph Laplacian problems, see~\cite{ColleyLinHuAeron2017,LinCowenHescottHu2018a,HuLinZikatanov2019}. 

To validate the efficiency of multi-grid for this problem, we include the solver performance on the three test cases from the manuscript in the tables below. 

The performance of the iterative solver for Case 1A is summarized in Table \ref{tab:Case1B-solver}.  The results for Case 1B are essentially identical and are omitted (the deviation is less than $5\%$ in all quantities) . This test case, due to the simplicity of the network, is from a linear algebra perspective almost identical to the 2D homogeneous Poisson problem.  Therefore, the aggregation-based AMG preconditioner performs as expected. Here, the grid complexity (GridComp.) is defined as the ratio between the total size of the matrices on different levels and the size of the matrix on the finest level.  The operator complexity (Oper.Comp.) is defined as the ratio between the total number of nonzeros of the matrices on different levels and the number of nonzeros of the matrix on the finest level.  As we can see, both Grid Complexity and Operator Complexity are bounded by $2$, which means that the aggregation-based AMG method roughly preserves the sparsity on the coarse levels and the overall computational cost of one V-cycle AMG method is optimal (i.e. linearly proportional to the cost of one step of Gauss-Seidel method).  Since we use a simple unsmoothed aggregation AMG method here, the number of iterations grows with respect to the size of the linear system as expected. Overall, the computational cost of the multilevel solver is sub-optimal.  We want to point out that it is possible to obtain optimal computational complexity when a more sophisticated AMG cycle, such as the K-cycle is used~\cite{Vassilevski2008,HuVassilevskiXu2013}. This is a subject of our ongoing and future research since developing special tailored multilevel solvers for mixed-dimensional problems is a challenging and important task itself and beyond the scope of this work.

\begin{table}[]
\scriptsize
\begin{center}
\begin{tabular}{ l  l  l  l  l  l  l  l }
$1/h$ & NLevel & GridComp. & Oper.Comp. & NIter & CPUsetup & CPUSolve & DOF\\
\hline
$16$ &     2 & 1.33 & 1.42 &     19 & 2.14e-03 & 1.54e-02 & 256\\
$32$ &     4 & 1.59 & 1.79 &     30 & 7.50e-03 & 2.32e-02 & 1024\\
$64$ &     5 & 1.66 & 1.90 &     44 & 2.47e-02 & 1.04e-01 & 4096\\
$128$ &     7 & 1.71 & 1.97 &     66 & 9.27e-02 & 3.60e-01 & 16384\\
$256$ &     8 & 1.73 & 2.00 &     95 & 2.96e-01 & 1.66e+00 & 65536\\
$512$ &     9 & 1.74 & 2.02 &    140 & 1.32e+00 & 9.26e+00 & 262144\\
\hline
\end{tabular}
\end{center}
\caption{Solver performance on Case 1A. (NLevel: number of levels; GridComp.: grid complexity; Oper. Comp.: operator complexity; NIter: number of iterations; CPUsetup: CPU time for setup AMG; CPU solve: CPU time for solving using AMG; DOE: degrees of freedoms.) }\label{tab:Case1B-solver}
\end{table}

The solver results for case 2 are presented in Table \ref{tab:Case2-solver}, including the performance on the reference solution, which has more than 33 million degrees of freedom.  Again, we observe that both grid complexity and operator complexity are bounded while the number of iterations grows sublinearly with respect to the size of the linear systems.  This means the overall computational complexity is sub-optimal, which is expected since we are using a simple unsmoothed aggregation AMG preconditioner.

\begin{table}[]
\scriptsize
\begin{center}
\begin{tabular}{ l  l  l  l  l  l  l }
$1/h$ & NLevel & GridComplexity & OperatorComplexity & NIter & CPUAMGsetup & CPUSolve\\
\hline
    $16$ &     06 & 1.49e+00 & 1.76e+00 &     24 & 4.53e-02 & 6.35e-02\\
    $32$ &     08 & 1.58e+00 & 1.93e+00 &     33 & 3.93e-01 & 5.39e-01\\
    $64$ &     10 & 1.63e+00 & 2.02e+00 &     43 & 3.96e+00 & 5.02e+00\\
    $128$ &     12 & 1.65e+00 & 2.07e+00 &     61 & 4.04e+01 & 8.55e+01\\
    $256$ &     14 & 1.66e+00 & 2.09e+00 &     83 & 3.72e+02 & 1.05e+03\\
\hline
\end{tabular}
\end{center}
\caption{Solver performance on Case 2. See \ref{tab:Case1B-solver} for complete caption information. } \label{tab:Case2-solver}
\end{table}

The computational performance on Case 3, which consists of real data, is summarized in Table \ref{tab:solver-full-brain}. All results are reported relative to the discretization level of the data, which we label $h_0$. The results confirm that even for this large data-set, with data-driven and physuiologically realistic parameters, the solver performs efficiently.  The number of iteration grows sublinearly and the CPU time also indicates that the current solver is sub-optimal.  However, we do not see the degradation of the AMG solver comparing with the previous two cases with simpler geometry, which demonstrates that our discretization is stable and solver-friendly. Moreover, the efficiency of the discretization and solver combinations proposed above allow for analysis of full-resolution medical data.

\begin{table}[]
\scriptsize
\begin{center}
\begin{tabular}{ l l l l l l l l }
$h_0/h$ & NLevel & GComp & OComp & NIter & CPUsetup & CPUSolve & DOF\\
\hline
0.064 &     05 & 1.27 & 1.13 &     68 & 4.14e-02 & 2.64e-01 &   2206\\
0.13 &     06 & 1.12 & 1.02 &    108 & 2.20e-01 & 2.74 &  17179\\
0.25 &     08 & 1.15 & 1.04 &    113 & 2.37 & 22.8 & 137378\\
0.51 &     10 & 1.27 & 1.09 &    122 & 23.4 & 200 & 1103295\\
1 &     12 & 1.53 & 1.36 &    138 & 317 & 2.74e+03 & 8792747\\
\hline
\end{tabular}
\end{center}
\caption{Solver report whole brain simulation. Note that we set $h_0$ as the resolution of the \textit{finest} grid, since this is the resolution of the data. Also, the DOF column reflects that due to the nontrivial shape of the brain, the number of active cells in the computation does not scale exactly with the cube of the grid resolution. }
\label{tab:solver-full-brain}
\end{table}

\section{Derivation of reference solution for section 5.1}

The solution is valid for any domain $\Omega$ containing the origin and with zero Neumann boundary conditions, and for a tree consisting of two nodes: Node 0 is a Dirichlet node and node 1 is a terminal node. Please refer to Figure \ref{fig:case1_illustration} for a visual illustration. 

\subsection*{Setup}

Consider polar coordinates $(r,\theta)$ around the origin, included in $\Omega$. In addition to Neumann boundary conditions on $\Omega$, the boundary pressure $p^N_0$ in the Dirichlet node is taken as a given, and we consider for simplicity zero source terms in the terminal node, $r^N_1=0$.

Then we construct up to four concentric domains, bounded by $0 < r_0 \leq r_1 \leq r_2 < r_3 \leq \mathrm{diam}\ (\Omega)$. The inner domain will be a disc, the remaining domains will be doughnuts. We consider homogeneous $k^D(\bm{x})=k^D$, while the remaing parameters in $\Omega$ are strictly functions of $r$. In particular, we choose a transfer permeability which is non-zero only in the two inner-most regions
\begin{equation}\label{eqn:kT-form}
k^T (r) = k^T_0
\begin{cases}
1 \quad \mathrm{for}\quad  r\leq r_0 , \\
a_0^2 \frac{r_1^2 -r^2}{r^2} \quad \mathrm{for}\quad  r\leq r_1, \\
0 \quad \mathrm{for}\quad  r > r_1.
\end{cases}
\end{equation}
Continuity of the transfer function at $r_0$ is obtained by setting 
\begin{equation}
a_0^2 = \frac{r_0^2}{r_1^2-r_0^2}
\end{equation}
The source is chosen as non-zero only between $r_2$ and $r_3$, 
\begin{equation}
r^D (r) = r^D_0 (r-r_2)^+ (r_3-r)^+,
\end{equation}
where $(a)^+\equiv \max (0,a)$ indicates that only positive values are considered, . 

The governing equations in the domain $\Omega$ can be stated in radial coordinates as (our parameters are radially symmetric, as will our solution be, and we drop $\theta$ in the continuation):
\begin{equation} \label{eqn:laplace-rad}
\frac{1}{r}\frac{d}{dr}\left(- r k^D \frac{dp^D}{dr} \right) - q^T = r^D
\end{equation}
where we recall that 
\begin{equation}\label{eqn:qT-rad}
q^T = -k^T (p^D-p^N_1) 
\end{equation}

\subsection{Solutions}

First, we note that by volume balance, we have that 
\begin{equation}
q^N \equiv q^N_{0,1} = -2\pi r_0^D \int_{r_2}^{r_3}r (r-r_2) (r_3-r)= -2\pi  r^D_0\left(\frac{r_3^4-r_2^4}{12} - r_2r_3\frac{r_3^2-r_2^2}{6}\right).
\end{equation}
It then follows that the pressure $p^N_1$ at the terminal node is given by
\begin{equation}
p^N_1 = p^N_0 - k^N q^N
\end{equation}

The combination of equations \eqref{eqn:laplace-rad} and \eqref{eqn:qT-rad} have analytical solutions within each region of the domain. These are as follows (these are reported in numerous text books, we will use Chapter 9 of Abramovich and Stegun as a concrete reference for the use and manipulation of Bessel functions \cite{abramowitz1948handbook}). 

\emph{Solution S1}: For $r\leq r_0$:  By substitution of equation \eqref{eqn:qT-rad} into equation \eqref{eqn:laplace-rad}, we obtain
\begin{equation} 
r\frac{d}{dr}\left(r\frac{dp^D}{dr} \right) - \frac{k^T_0}{k^D} r^2 (p^D-p^N_1) = 0,
\end{equation}
the solution of which is the modified Bessel function of the first kind $I_0$, since we easily determine that there is no singularity at the origin: 
\begin{equation}
p^D(r) - p^N_1 = c_I I_0 (\kappa_0 r) 
\end{equation}
where 
\begin{equation}
\kappa_0 = \sqrt{\frac{k^T_0}{k^D}}
\end{equation}
The constant $c_I$ will be determined later.  The derivative of presssure is obtained according to according to the differential rules for Bessel functions of integer order (see equation 9.6.27 in \cite{abramowitz1948handbook1}): 
\begin{equation}
\left. I_0'\right|_{\kappa_0 r_0}  = I_1(\kappa_0 r_0)
\end{equation}
Thus the radial component of flux is given by
\begin{equation}
q_r^D = - k^D c_I \kappa_0 I_1(\kappa_0 r_0)
\end{equation}

\emph{Solution S2}: For $r_0< r\leq r_1$:  By substitution of equation \eqref{eqn:qT-rad} into equation \eqref{eqn:laplace-rad}, we obtain in this region
\begin{equation} 
r \frac{d}{dr}\left(r\frac{dp^D}{dr} \right) + \frac{k^T_0}{k^D}a_0^2 (r^2-r_1^2) (p^D-p^N_1) = 0,
\end{equation}
the solution of which are the (unmodified) Bessel functions of the first and second kind $J$ and $Y$, of fractional order: 
\begin{equation}
p^D(r) - p^N_1 = c_J J_{\kappa_1 r_1} (\kappa_1 r) + c_Y Y_{\kappa_1 r_1} (\kappa_1 r) = \sum_{\mcC=J,Y}c_\mcC \mcC_{\kappa_1 r_1} (\kappa_1 r)
\end{equation}
where $\mcC$ indexes the Bessel functions $J$ and $Y$, and 
\begin{equation}
\kappa_1 = a_0 \sqrt{\frac{k^T_0}{k^D}}
\end{equation}
The constants $c_J$ and $c_Y$ will be determined later. From the pressure, we determine the derivative according to 
(see equation 9.1.30 in \cite{abramowitz1948handbook})
\begin{equation}
\mcC_\nu' (z) = \mcC_{\nu-1}(z) - \nu z^{-1} \mcC_{\nu}(z)
\end{equation}
thus the radial component of flux is given by
\begin{equation}
q_r^D = - k^D \kappa_1 \sum_{\mcC=J,Y}c_\mcC  \left(\mcC_{\kappa_1 r_1-1} (\kappa_1 r) - \frac{r_1}{r} \mcC_{\kappa_1 r_1} (\kappa_1 r)\right)
\end{equation}

\emph{Solution S3}: For $r_1< r\leq r_2$:  By equation \eqref{eqn:laplace-rad}
\begin{equation} 
\frac{d}{dr}\left(r\frac{dp^D}{dr} \right) = 0,
\end{equation}
the solution of which is the logarithm, and since $k^T(r)=0$ for $r>r_1$, we know that the total flux around any perimeter enclosing $r_1$ is given by $q^N$. This determines the constant for the logarithmic term, and we obtain:
\begin{equation}
p^D(r) - p^D(r_1) = - \frac{q^N}{2\pi k^D} \ln \left(\frac{r}{r_1}\right)
\end{equation}
The radial component of flux is given by
\begin{equation}
q_r^D =  \frac{q^N }{2\pi } \frac{1}{r}
\end{equation}

\emph{Solution S4}: For $r_2< r\leq r_3$:  By equation \eqref{eqn:laplace-rad}
\begin{equation} 
\frac{1}{r}\frac{d}{dr}\left(r\frac{dp^D}{dr} \right) = - \frac{r^D_0}{k^D} (r-r_2) (r_3-r),
\end{equation}
the solution of which is a fourth-order polynomial plus a logarithmic term.
\begin{equation}
p^D(r) - p^D(r_2) = c_4 \ln \left(\frac{r}{r_2}\right) + \frac{r^D_0}{k^D} \left(\frac{r^4-r_2^4}{16}-(r_3+r_2)\frac{r^3-r_2^3}{9} +r_2r_3\frac{r^2-r_2^2}{4}\right)
\end{equation}
The radial component of flux is given by
\begin{equation}
q_r^D = - k^D\frac{c_4}{r} - 
r^D_0 \left(\frac{r^3}{4}-(r_3+r_2)\frac{r^2}{3} +r_2r_3\frac{r}{2}\right)
\end{equation}
However, since the outer boundary condition of $\Omega$ is zero Neumann, and we have continuity of flux with respect to the solution S4 at $r_3$, this implies that we can determine the constant multiplier of the logarithmic term by setting $q_r^D(r_3)=0$, which yields:
\begin{equation}
c_4 = \frac{r^D_0}{k^D} \left(\frac{r_3^4}{12}-\frac{r_2r_3^3}{6}\right)
\end{equation}

\emph{Solution S5}: For $r_3 < r$:  By equation \eqref{eqn:laplace-rad}
\begin{equation} 
\frac{d}{dr}\left(r\frac{dp^D}{dr} \right) = 0,
\end{equation}
the solution of which is again the logarithm as in S3. However, in accordance with the zero Neumann boundary condition, the flux is zero, and all that remains is a constant. 
\begin{equation}
p^D(r) - p^D(r_3) = 0
\end{equation}
and
\begin{equation}
q_r^D = 0
\end{equation}

\subsection*{Constants}
It remains to determine the constants $c_I$, $c_J$ and $c_Y$ from solutions S1 and S2. The solutions S1 and S2 must satisfy three criteria: 1) The pressure must be continuous at $r_0$, 2) the flux must be continuous at $r_0$, and 3) The flux must be continuous at $r_1$. This leads to three linear constraints on the constants.

These criteria can be made precise using the solutions derived above as follows.

\emph{2) Pressure continuity}  Continuity of pressure at $r_0$ requires that
\begin{equation}\label{eqn-const2}
c_I I_0 (\kappa_0 r_0)  = \sum_{\mcC=J,Y}c_\mcC \mcC_{\kappa_1 r_1}(\kappa_1 r_0) 
\end{equation}

\emph{3) Flux continuity at $r_0$}  Continuity of flux at $r_0$ requires that
\begin{equation}\label{eqn-const3}
c_I \kappa_0 I_1(\kappa_0 r_0)  = \kappa_1 \sum_{\mcC=J,Y}c_\mcC  \left(\mcC_{\kappa_1 r_1-1} (\kappa_1 r_0) - \frac{\kappa_1 r_1}{\kappa_1 r_0} \mcC_{\kappa_1 r_1} (\kappa_1 r_0)\right)
\end{equation}

\emph{3) Flux continuity at $r_1$} Continuity of flux at $r_0$ requires that
\begin{equation}\label{eqn-const1}
\frac{q^N }{2\pi} \frac{1}{r_1}
= - k^D\kappa_1 \sum_{\mcC=J,Y}c_\mcC  \left(\mcC_{\kappa_1 r_1-1} (\kappa_1 r_1) - \mcC_{\kappa_1 r_1} (\kappa_1 r_1)\right)
\end{equation}

The three equations \eqref{eqn-const1}, \eqref{eqn-const2} and \eqref{eqn-const3} provide the three linear constraints for the unknown constants. 

For the special case where $r_0 = r_1$, two of the constants, $c_J$ and $c_Y$ are superfluous, and the remaining constant is obtained by combining equations \eqref{eqn-const2} and \eqref{eqn-const3} as:

\begin{equation}
c_I k^D \kappa_0 I_1(\kappa_0 r_1)  = -\frac{q^N }{2\pi} \frac{1}{r_1}
\end{equation}

\section{Full-size figures for section 5.2}

This section contains enlarged versions of the subfigures in Figure 5 of the article, which contains the colorbars for the figures in this section. 

\begin{figure}[H]\label{fig:SM1}
    \centering
    \includegraphics[width=12cm]{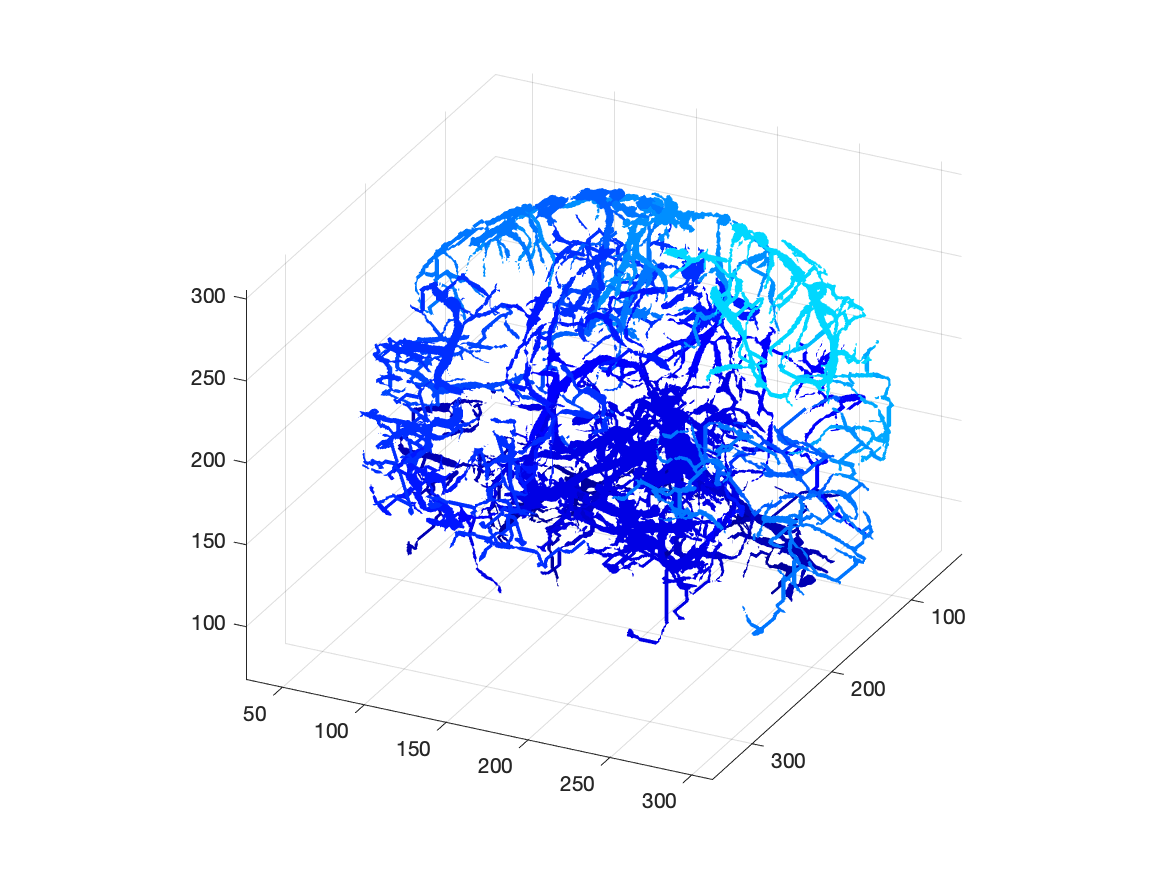}
    \includegraphics[width=12cm]{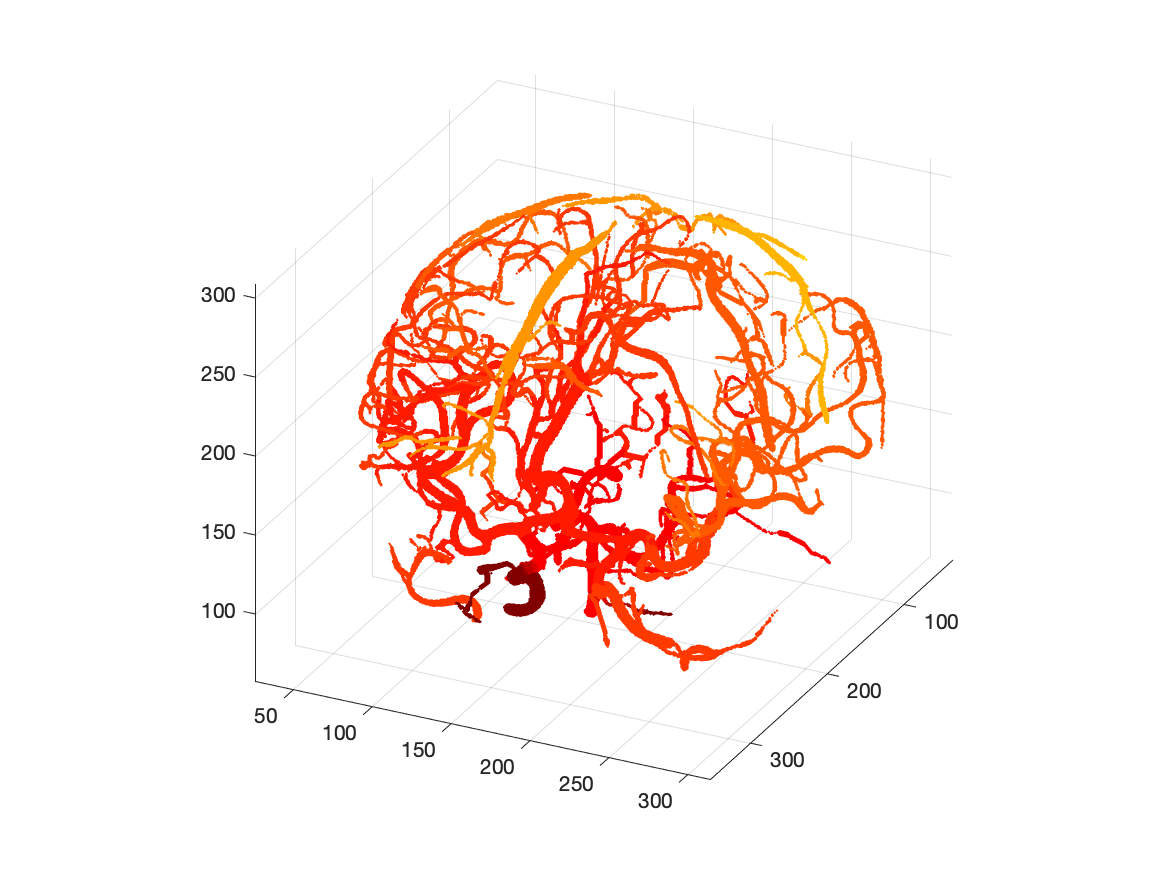}
    \caption{Simulation results for Case 3, showing the pressure solution for venous (upper) and arterial (lower) trees, $p^N$}
\end{figure}

\begin{figure}[H]\label{fig:SM2}
    \centering
    \includegraphics[width=12cm]{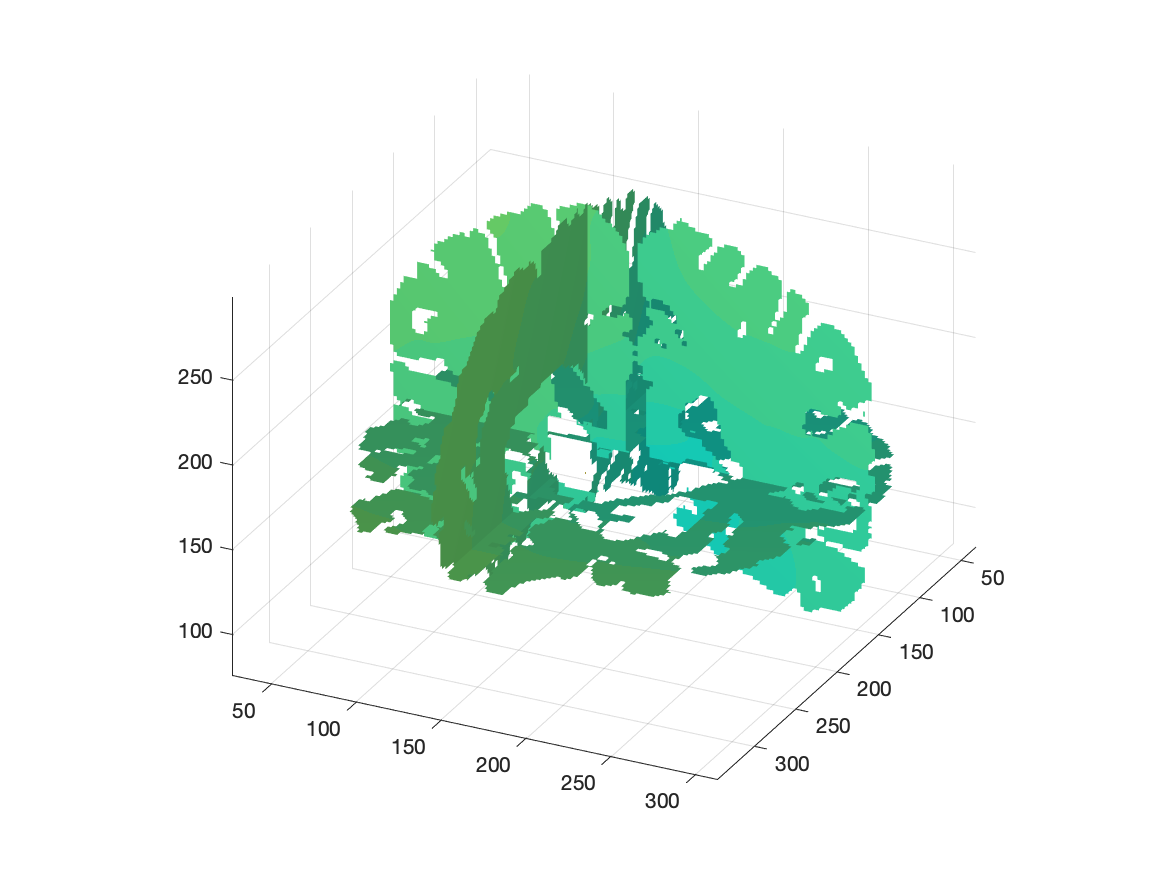}
    \includegraphics[width=12cm]{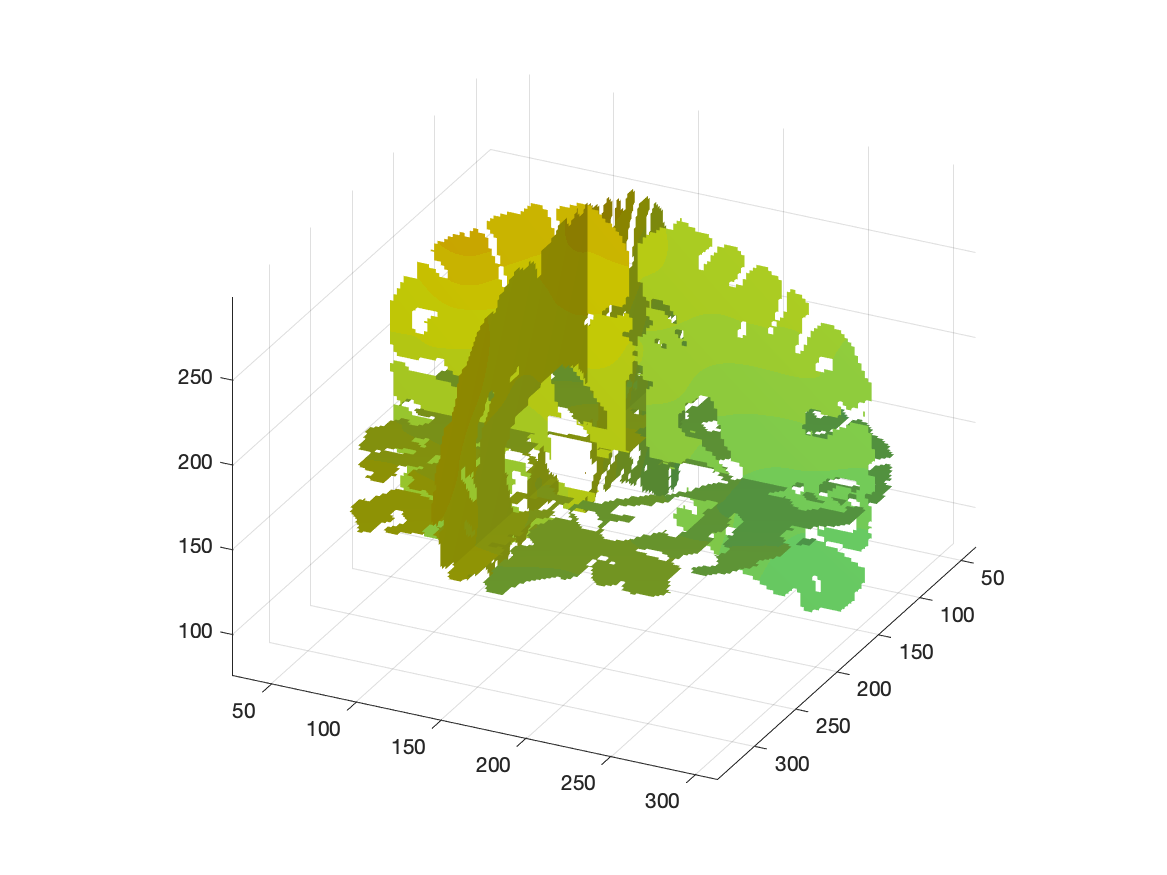}
    \caption{Simulation results for Case 3, showing the pressure solution for venous (upper) and arterial (lower) parts of the domain, $p^D$}
\end{figure}

\begin{figure}[H]
    \centering
    \includegraphics[width=9cm]{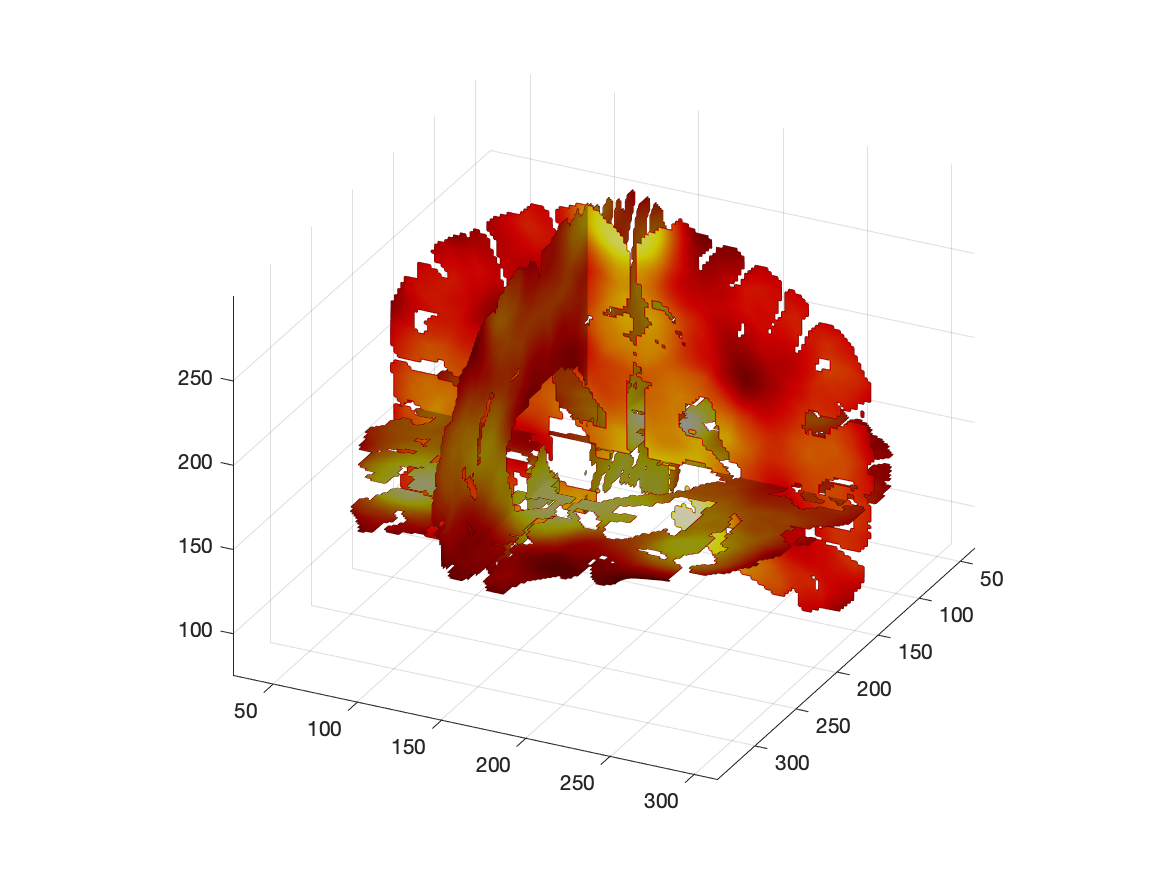}
    \includegraphics[width=9cm]{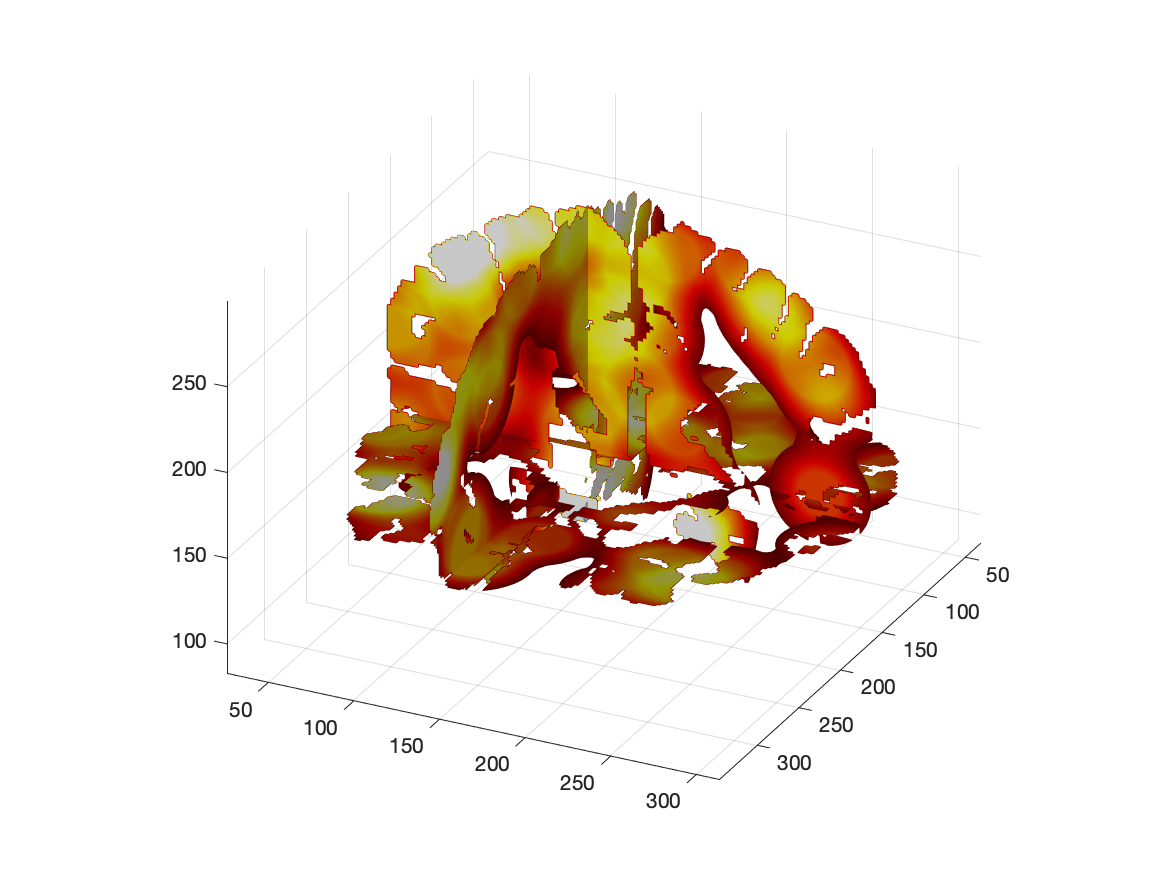}
    \includegraphics[width=9cm]{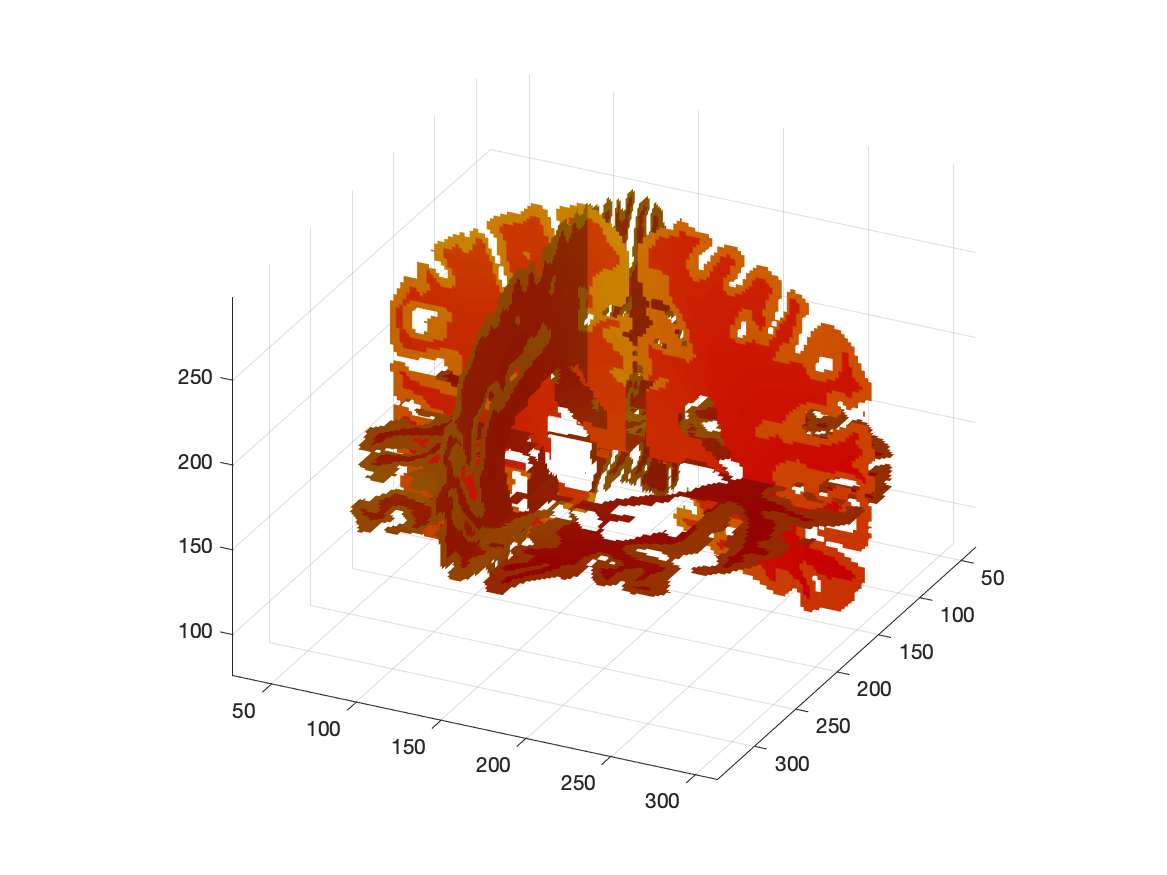}
    \caption{Simulation results for Case 3, showing the fluxes associated with the venous (upper) and arterial (middle right) unresolved trees, $q^T$, as well as perfusion $q^P$ (lower).}
\end{figure}


\bibliographystyle{siamplain}
\bibliography{BloodFlowBrainSI.bib}